\documentclass[12pt]{amsart}
\usepackage{amssymb, amsmath}
\usepackage[all]{xy}

\theoremstyle{plain}
\newtheorem{theorem}{Theorem}[section]

\newtheorem{cor}[theorem]{Corollary}
\newtheorem{lemma}[theorem]{Lemma}
\newtheorem{proposition}[theorem]{Proposition}
\newtheorem{prop}[theorem]{Proposition}

\theoremstyle{definition}
\newtheorem{definition}[theorem]{Definition}

\theoremstyle{remark}
\newtheorem{remark}[theorem]{Remark}
\newtheorem{remarks}[theorem]{Remarks}

\newtheorem{example}[theorem]{Example}

\newtheorem*{notationnonumber}{Notation}
\numberwithin{equation}{section}

\DeclareMathOperator{\End}{End} 
\DeclareMathOperator{\KP}{KP}
\DeclareMathOperator{\BD}{BD}
\DeclareMathOperator{\lsp}{span}
\DeclareMathOperator{\rank}{rank}
\DeclareMathOperator{\clsp}{\overline{span}}
\DeclareMathOperator{\Aut}{Aut}
\newcommand{\Ind}{\operatorname{\mathrm{Ind}}}
\newcommand{\Res}{\operatorname{\mathrm{Res}}}
\newcommand{\id}{\operatorname{\mathrm{id}}}

\def\C{{\mathcal C}}

\renewcommand{\L}{\Lambda}
\renewcommand{\O}{\Omega}
\def\N{{\mathbb N}}
\def\Z{{\mathbb Z}}
\def\CC{{\mathbb C}}
\def\TT{{\mathbb T}}
\def\QQ{{\mathbb Q}}
\def\l{{\lambda}}

\def\Ll{\mathcal{L}}
\def\Ff{\mathcal{F}}

\def\FF{\mathbb{F}}

\begin{document}
\title[Kumjian-Pask algebras of higher-rank graphs]{Kumjian-Pask algebras of higher-rank graphs}
\author{Gonzalo Aranda Pino}
\address{Gonzalo Aranda Pino: Departamento de \'Algebra, Geometr\'ia y Topolog\'ia, Universidad de M\'alaga, 29071 M\'alaga, Spain}
\email{g.aranda@uma.es}
\author{John Clark}
\author{Astrid an Huef}
\author{Iain Raeburn}
\address{John Clark, Astrid an Huef and Iain Raeburn: Department of Mathematics and Statistics, University of Otago, PO Box 56, Dunedin 9054, New Zealand.}
\email{\{jclark, astrid, iraeburn\}@maths.otago.ac.nz}
\thanks{The results in this paper were obtained during a working seminar at the University of Otago. The authors thank the other participants for their comments and input, and especially Jon Brown, Iain Dangerfield and Robbie Hazlewood.}
\date{22 June 2011}
\begin{abstract}
We introduce higher-rank analogues of the Leavitt path algebras, which we call the Kumjian-Pask algebras. We prove graded and Cuntz-Krieger uniqueness theorems for these algebras, and  analyze their ideal structure.  
\end{abstract}


\maketitle

\section{Introduction}

The $C^*$-algebras of infinite directed graphs were first studied in the 1990s \cite{KPRR, KPR} as generalizations of the Cuntz-Krieger algebras of finite $\{0,1\}$-matrices \cite{CK}. The Leavitt path algebras, which are a purely algebraic analogue of graph $C^*$-algebras, were first studied around 2005 \cite{AA1,AMP}. Both families of algebras have been intensively studied by a broad range of researchers, both now have substantial structure theories, and both have proved to be rich sources of interesting examples.

Higher-rank analogues of the Cuntz-Krieger algebras arose first in work of Robertson and Steger \cite{RobS1,RobS2}, and shortly afterwards Kumjian and Pask  \cite{KP} introduced higher-rank graphs (or \emph{$k$-graphs}) to provide a visualisable model for Robertson and Steger's algebras. The higher-rank graph $C^*$-algebras constructed in \cite{KP} have generated a great deal of interest among operator algebraists (for example, \cite{DPY, KrP, P, SZ,Y}), and have broadened the class of $C^*$-algebras that can be realized as graph algebras \cite{DY,KP,PRRS,PRW}. Here we introduce and study an analogue of Leavitt path algebras for higher-rank graphs. We propose to call these new algebras the \emph{Kumjian-Pask algebras}.

For operator algebraists, there is a well-trodden path for studying new analogues of Cuntz-Krieger algebras, which was developed in \cite{aHR} and \cite{BPRS}, and which was followed in the first four chapters of \cite{R}, for example. First, one constructs an algebra which is universal among $C^*$-algebras generated by families of partial isometries satisfying suitable Cuntz-Krieger relations. Next, one proves uniqueness theorems which say when a representation of this algebra is injective: there should be a gauge-invariant uniqueness theorem, which works without extra hypotheses on the graph, and a Cuntz-Krieger uniqueness theorem, which has a stronger conclusion but requires an aperiodicity hypothesis. And then one hopes to use these theorems to analyze the ideal structure.

Tomforde tramped this path for the Leavitt path algebras over fields \cite{T}, and more recently has retramped it for Leavitt path algebras over commutative rings \cite{T2}. We will use the same path to study the Kumjian-Pask algebras of row-finite $k$-graphs without sources. There are satisfactory $C^*$-algebraic uniqueness theorems for larger families of $k$-graphs, but they can be very complicated to work with (look at the proof of Cuntz-Krieger uniqueness in \cite{RSY2}, for example). So for a first pass it seems sensible to stick to the row-finite case, which covers most of the interesting examples. We follow \cite{T2} in allowing coefficients in an arbitrary commutative ring $R$ with identity $1$.

We begin with a section on background material. We recall from \cite{KP} some elementary facts about higher-rank graphs and their infinite path spaces, and also discuss some basic properties of gradings on free algebras which we couldn't find in suitable form in the literature. Then in \S\ref{sec:KPfam}, we describe our Kumjian-Pask relations for a row-finite $k$-graph $\L$ without sources, and construct the Kumjian-Pask algebra $\KP_R(\L)$ as a quotient of the free $R$-algebra on the set of paths in $\Lambda$. Because the Kumjian-Pask relations are substantially more complicated for $k$-graphs, we have had to be quite careful with this construction, and in particular with the existence of the $\Z^k$-grading on $\KP_R(\L)$. 

In \S\ref{sectionuniqueness}, we prove a graded-uniqueness theorem and a Cuntz-Krieger uniqueness theorem for $\KP_R(\L)$. We have used similar arguments to those of \cite[\S5--6]{T2}, but, partly because we are only interested in the row-finite case, we have been able to streamline the arguments and find a common approach to the two theorems. In particular, we were able to bypass the complicated induction arguments used in \cite{T2}. As the main hypothesis in our  Cuntz-Krieger uniqueness theorem we use the ``finite-path formulation'' of aperiodicity due to Robertson and Sims \cite{RS}.

In \S\ref{sectionbasicsimplicity} and \ref{sectionsimple}, we investigate the ideal structure of $\KP_R(\L)$. The first step is to describe the graded ideals, which we do in Theorem~\ref{latticehersat}; as in \cite{T2}, to get the usual description of ideals in terms of saturated hereditary subsets of vertices (which goes back to Cuntz \cite{C}), we have to restrict attention to a class of ``basically ideals''. We then give an analogue of Conditions (II) of \cite{C} and (K) of \cite{KPRR} which ensures that every basic ideal is graded, and describe the $k$-graphs for which $\KP_R(\L)$ is ``basic simple''. Then in \S\ref{sectionsimple}, we find necessary and sufficient conditions for $\KP_R(\L)$ to be simple in the more conventional sense. This last result is new even for $1$-graphs.  We discuss examples and applications in \S\ref{sectionexample}.


\section{Background}

 We write $\N$ for the set of natural numbers, including $0$. Let $k$ be a positive integer. For $m,n\in \N^k$, $m\leq  n$ means $m_i\leq n_i$ for $1\leq i\leq k$ and $m\vee n$ denotes the pointwise maximum. We denote the usual basis in $\N^k$ by  $\{e_i\}$.

 In a category $\C$ with objects $\C^0$, we identify objects $v\in\C^0$ with their identity morphisms $\iota_v$, and write $\C$ for the set of morphisms; we write $s$ and $r$ for the domain and codomain maps from $\C$ to $\C^0$, and usually denote the composition of morphisms by juxtaposition.
 
 A directed graph $E=(E^0,E^1,r,s)$ consists of countable sets $E^0$ and $E^1$  and functions $r,s:E^1\to E^0$. As usual, we think of the elements of $E^0$ as vertices and the elements $e$ of $E^1$ as edges from $s(e)$ to $r(e)$. Because we are going to be talking about higher-rank graphs, where a juxtaposition $\mu\nu$ stands for the composition of morphisms $\mu$ and $\nu$ with $s(\mu)=r(\nu)$, we use the conventions of \cite{R} for paths in $E$. Thus a path of length $|\mu|:=n$ in $E$ is a string $\mu=\mu_1\cdots\mu_n$ of edges $\mu_i$ with $s(\mu_i)=r(\mu_{i+1})$ for all $i$, and the path has source $s(\mu):=s(\mu_n)$ and range $r(\mu):=r(\mu_1)$. 
 
\subsection{Higher-rank graphs}

For a positive integer $k$, we view the additive semigroup $\N^k$ as a category with one object. Following Kumjian and Pask \cite{KP}, a \emph{graph of rank $k$} or \emph{$k$-graph}
is a countable category $\L = (\Lambda^0,\L,r,s)$ together with a functor $d:\L \to {\N}^k$, called
\emph{the degree map}, satisfying the following \emph{factorization property}:
if $\lambda \in \L$ and $d(\lambda) = m+n$ for some $m, n \in {\mathbb N}^k$, then there are unique $\mu,
\nu \in \L$ such that $d(\mu) = m$, $d(\nu) = n$, and $\lambda = \mu\nu$. 

The motivating example is:

\begin{example} Consider a directed graph $E=(E^0,E^1,r,s)$. Then the \emph{path category} $P(E)$ has object set $E^0$, and the morphisms in $P(E)$ from $v\in E^0$ to $w\in E^0$ are finite paths $\mu$ with $s(\mu)=v$ and $r(\mu)=w$;  composition is defined by concatenation, and the identity morphisms obtained by viewing the vertices as paths of length $0$. With the degree functor $d:\mu\mapsto |\mu|$, the path category $(P(E),d)$ is a $1$-graph.
\end{example}

With this example in mind, we make some conventions.  If $\lambda\in\L$ satisfies $d(\lambda)=0$, the identities $\iota_{r(\lambda)}\lambda=\lambda=\lambda\iota_{s(\lambda)}$ and the factorization property imply that $\iota_{r(\lambda)}=\lambda=\iota_{s(\lambda)}$; thus $v\mapsto \iota_v$ is a bijection of $\Lambda^0$ onto $d^{-1}(0)$. Then for $n\in \N^k$, we write $\L^n:=d^{-1}(n)$, and call the elements $\lambda$ of $\Lambda^n$ \emph{paths of degree $n$ from $s(\lambda)$ to $r(\lambda)$}. For $v\in \Lambda^0$ we write $v\Lambda^n$ or $v\Lambda$ for the sets of paths with range $v$ and $\Lambda^nv$ or $\Lambda v$ for paths with source $v$.

To visualise a $k$-graph $\L$, we think of the object set $\Lambda^0$ as the vertices in a directed graph, choose $k$ colours $c_1, \ldots,c_k$, and then for each $\lambda\in \Lambda^{e_i}$, we draw an oriented edge of colour $c_i$ from $s(\lambda)$ to $r(\lambda)$. This coloured directed graph $E$ is called the \emph{skeleton} of $\L$. When $k=1$, the skeleton is an ordinary directed graph, and completely determines the $1$-graph: indeed, the factorization property allows us to write each morphism $\lambda$ of degree $n$ uniquely as the composition $\lambda_1\circ\lambda_2\circ\cdots \circ\lambda_n$ of $n$ morphisms of degree 1, and then the map which takes $\lambda$ to the path $\lambda_1\lambda_2\cdots\lambda_n$ is an isomorphism of $\L$ onto $P(E)$.
When $k>1$, the skeleton does not always determine the $k$-graph. To discuss this issue, we need some examples.

\begin{example} Let $\O_k^0:=\N^k$, $\O_k:=\{(p,q)\in \N^k\times\N^k : p\leq q\}$, define $r,s:\O_k\to \O_k^0$ by $r(p,q):=p$ and $s(p,q):=q$, define composition by $(p,q)(q,r)=(p,r)$, and define $d:\O_k\to\N^k$ by $d(p,q):=q-p$. Then $\O_k=(\O_k,r,s,d)$ is a $k$-graph.

Similarly, for $m\in \N^k$  we define $\O_{k,m}^0:=\{p\in \N^k : p\leq m\}$ and $\O_{k,m}=\{(p,q)\in \O_{k,m}^0\times
\O_{k,m}^0 : p\leq q\}$, and then with the same $r$, $s$ and $d$, $\O_{k,m}$ is a $k$-graph. The skeleton of $\O_{2,(3,2)}$, for example, is
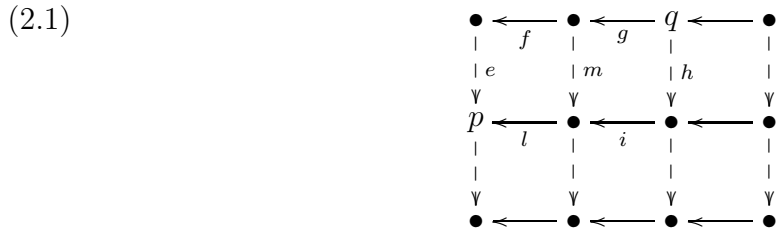
\begin{equation}\label{weeskeleton}
\xymatrix{
 \bullet \ar@{-->}[d]^e & \bullet \ar@{-->}[d]^m \ar[l]^f &
  q \ar@{-->}[d]^h \ar[l]^g&
 \bullet \ar[l] \ar@{-->}[d] \\
 p \ar@{-->}[d] & \bullet \ar@{-->}[d] \ar[l]^l & \bullet \ar@{-->}[d] \ar[l]^i &
 \bullet \ar[l] \ar@{-->}[d]   \\
 \bullet & \bullet \ar[l] & \bullet \ar[l] & \bullet \ar[l] }
\end{equation}
where the solid arrows are blue, say, and the dashed ones are red. We think of the paths as rectangles: for example, the path $(p,q)$ with source $q$ and range $p$ would be the $2\times 1$ rectangle in the top left, and the different routes $efg$, $lmg$, $lih$ from $q$ to $p$ represent the different factorizations of $(p,q)$.
\end{example}

The factorization property in a $k$-graph $\L$ sets up bijections between the $c_ic_j$-coloured paths of length 2 and the $c_jc_i$-coloured paths, and we think of each pair as a commutative square in the skeleton. A theorem of Fowler and Sims \cite{FS} says that this collection $\C$ of commutative squares determines the $k$-graph; a path of degree $(3,2)$, for example, is obtained by pasting a copy of \eqref{weeskeleton} round the skeleton in such a way that the colours are preserved and each constituent square is commutative. When $k=2$, every collection $\C$ which includes each $c_ic_j$-coloured path exactly once determines a $2$-graph with the given skeleton \cite[\S6]{KP}; for $k\geq 3$, the collection $\C$ has to satisfy an extra associativity condition. (For a discussion with some pictures, see \cite{RSY}.)

A $k$-graph $\L$ is \emph{row-finite} if
$v\Lambda^n$ is finite for every $v\in\L^0$ and $n\in{\mathbb N}^k$; $\L$ \emph{has no sources} if
$v\Lambda^n$ is nonempty for every $v\in\L^0$ and $n\in{\mathbb N}^k$. In this paper we are only interested  in row-finite $k$-graphs without sources. 

\subsection{The infinite path space}

Suppose that $\L$ is a row-finite $k$-graph without sources. Following \cite[\S2]{KP}, an \emph{infinite path} in $\L$ is a degree-preserving functor $x:\O_k\to \L$. We denote the set of all infinite paths by $\L^\infty$. Since we identify the object $m\in \O_k$ with the identity morphism $(m,m)$ at $m$, we write $x(m)$ for the vertex $x(m,m)$. Then the range of an infinite path $x$ is the vertex $r(x):=x(0)$, and we write $v\Lambda^\infty:=r^{-1}(v)$.

\begin{remark}
To motivate this definition, notice that an ordinary path $\lambda\in \L^n$ gives a functor $f_\l:\O_{k,n}\to\L$. To see this, take $0\leq p\leq q\leq n$, use the factorization property to see that there are unique paths $\l'\in\L^p$, $\l''\in \L^{q-p}$ and $\l'''\in \L^{n-q}$ such that $\l=\l'\l''\l'''$, and then define $f_\l(p,q):=\lambda(p,q):=\l''$. The map $\l\mapsto f_\lambda$ is a bijection from $\L^n$ onto the set of degree-preserving functors from $\O_{k,n}$ to $\L$ \cite[Examples 2.2(ii)]{RSY}.  
\end{remark}

Since $\Lambda$ has no sources, every vertex receives paths of arbitrarily large degrees, and the following lemma from \cite{KP} tells us that every vertex receives infinite paths.

\begin{lemma} \label{lemma1infinitepaths} \cite[Remarks 2.2]{KP} Suppose that $n(i)\leq n(i+1)$ in $\N^k$, that $n(i)_j\to \infty$ in $\N$ for $1\leq j\leq k$, and that $\l_i\in \L^{n(i)}$ satisfy
$\l_{i+1}(0,n(i))=\l_i$. Then there is a unique $y\in \L^\infty$ such that $y(0,n(i))=\l_i$.
\end{lemma}

The next lemma, also from \cite[\S2]{KP}, tells us that we can compose infinite paths with finite ones, and that there is a converse factorization process. The path $x(n,\infty)$ in part (b) was denoted $\sigma^n(x)$ in \cite{KP}.

\begin{lemma}\label{lemma2infinitepaths} \textnormal{(a)} Suppose that $\l\in \L$ and $x\in \L^\infty$ satisfy $s(\lambda)=r(x)$. Then there is a unique $y\in \L^\infty$ such that $y(0,n)=\lambda x(0,n-d(\lambda))$ for $n\geq d(\l)$; we then write $\lambda x:=y$.

\textnormal{(b)} For $x\in \L^\infty$ and $n\in \N^k$, there exist unique $x(0,n)\in \L^n$ and $x(n,\infty)\in \L^\infty$ such that $x=x(0,n)x(n,\infty)$.
\end{lemma}

\subsection{Graded rings}

Let $G$ be an additive abelian group. A ring $A$ is \emph{$G$-graded} if there are additive subgroups $\{A_g :g\in G\}$ of $A$ such that $A_gA_h\subset A_{g+h}$ and every nonzero $a\in A$ can be written in exactly one way as a finite sum $\sum_{g\in F}a_g$ of nonzero elements $a_g\in A_g$. The elements of $A_g$ are \emph{homogeneous of degree $g$}, and $a=\sum_{g\in F}a_g$ is the \emph{homogeneous decomposition} of~$a$. If $A$ and $B$ are $G$-graded rings, a homomorphism $\phi:A\to B$ is \emph{graded} if $\phi(A_g)\subset B_g$ for all $g\in G$. 

Suppose that $A$ is $G$-graded by $\{A_g:g\in G\}$. An ideal $I$ in $A$ is a \emph{graded ideal} if $\{I\cap A_g:g\in G\}$ is a grading of $I$. If $I$ is graded and $q:A\to A/I$ is the quotient map, then $A/I$ is $G$-graded by $\{q(A_g):g\in G\}$. To check that an ideal $I$ is graded, it suffices (by the uniqueness of homogeneous decompositions in $A$) to check that every element of $I$ is a sum of elements in $\bigcup_{g\in G}(I\cap A_g)$. Every ideal $I$ which is generated by a set $S$ of homogeneous elements is graded: to see this, it suffices by linearity and the previous observation to check that every element of\[\{a_gxb_h:a_g\in A_g,\ x\in S,\ b_h\in A_h\}\]belongs to some $I\cap A_k$, and this is easy.

For a nonempty set $Y$, we view the free $R$-module $\FF_R(Y)$ with basis $Y$ as the set of formal sums $\sum_{y\in Y}r_yy$ in which all but finitely many coefficients $r_y$ are zero; we view the elements $y\in Y$ as elements of $\FF_R(Y)$ by writing them as sums $\sum_x r_xx$ where $r_x=0$ for $x\not= y$ and $r_y=1$. For a nonempty set $X$, we let $w(X)$ be the set of words $w$ from the alphabet $X$, and we write $|w|$ for the length of $w$, so that $w=w_1w_2\cdots w_{|w|}$ for some $w_i\in X$. Then the free $R$-module ${\mathbb F}_R(w(X))$ is an $R$-algebra with multiplication given by
\begin{equation}\label{defprodfree}
\Big(\sum_{w\in w(X)}r_ww\Big)\Big(\sum_{y\in w(X)}s_yy\Big)=
\sum_{z\in w(X)}\Big(\sum_{\{w,y\in w(X)\,:\,wy=z,\; r_w\not=0,\;s_y\not=0\}}r_ws_y\Big)z.
\end{equation}
This algebra is the free $R$-algebra on $X$:

\begin{prop}\label{univfreealg}
The elements of $X$ generate $\FF_R(w(X))$ as an $R$-algebra. Suppose that $f$ is a function from $X$ into an $R$-algebra $A$. Then there is an $R$-algebra homomorphism $\phi_f:\FF_R(w(X))\to A$ such that
\begin{equation}\label{defphid}
\phi_f\Big(\sum_{w\in w(X)}r_ww\Big)=\sum_{r_w\not=0}r_wf(w_1)f(w_2)\cdots f(w_{|w|}).
\end{equation}
\end{prop}

\begin{proof}
Since each word $w$ is a product of $\{w_i:1\leq i\leq |w|\}$ and each $w_i\in X$, it is clear that $X$ generates $\FF_R(w(X))$ as an algebra. We can extend $f$ to a function on $w(X)$ by setting $f(w)=f(w_1)f(w_2)\cdots f(w_{|w|})$. Then the universal property of the free module $\FF_R(w(X))$ gives a well-defined $R$-module homomorphism $\phi_f:\FF_R(w(X))\to A$ satisfying \eqref{defphid}. Now a straightforward calculation using \eqref{defprodfree} shows that $\phi_f$ is an $R$-algebra homomorphism.
\end{proof}

We will want to put gradings on our free $R$-algebras, and the next proposition tells us how to do this. 

\begin{prop}\label{gradefree}
Suppose that $X$ is a nonempty set and $d$ is a function from $X$ to an additive abelian group $G$. Then there is a $G$-grading on $\FF_R(w(X))$ such that
\[
\FF_R(w(X))_g=\Big\{\sum_{w\in w(X)}r_w w:r_w\not=0\Longrightarrow d(w):=\sum_{i=1}^{|w|}d(w_i)=g\Big\}\quad \text{for $g\in G$.}
\]
\end{prop}

\begin{proof}
It is straightforward that each $\FF_R(w(X))_g$ is an additive subgroup of $\FF_R(w(X))$. To see that they span, consider $a=\sum_{w\in w(X)}r_ww \in {\FF}_R(w(X))$, and let $H:=\{w:r_w\not=0\}$. For $g\in G$ and $w\in w(X)$, we define
\[
s_{g, w} =
\begin{cases}
r_w & \text{ if $d(w) = g$} \\
0 & \text{ otherwise;}\\
\end{cases}
\]
then $a_g:=\sum_{w\in w(X)}s_{g,w}w$ belongs to $\FF_R(w(X))_g$, and $\sum_{g\in d(H)}a_g$ is a finite sum which is easily seen to be $a$. To show that the $\FF_R(w(X))_g$ are independent, suppose that $F$ is a finite subset of $G$, $a_g\in \FF_R(w(X))_g$ and $\sum_{g\in F}a_g=0$. Write  $a_g=\sum_{w\in w(X)} t_{g, w}w$. Then $t_{g,w}=0$ unless $g=d(w)$, and 
\[
0 = \sum_{g\in F}\sum_{w\in w(X)} t_{g, w}w = \sum_{w\in w(X)}\Big(\sum_{g\in F}t_{g, w}\Big)w = \sum_{w\in w(X)}t_{d(w), w}w.
\]
Then,  since the $0$ element of ${\mathbb F}_R(X)$ is the sum in which all coefficients are $0$, we get $t_{d(w), w} = 0$ for $w\in w(X)$. Thus we have $t_{g, w}
= 0$ for all $g, w$, and $a_g=0$ for all $g\in F$.

To see that $\FF_R(w(X))_g\FF_R(w(X))_h\subset \FF_R(w(X))_{g+h}$, we take $\sum_{w\in w(X)}r_ww$ in $\FF_R(w(X))_g$ and $\sum_{y\in w(X)}s_yy$ in $\FF_R(w(X))_h$, and multiply them using \eqref{defprodfree}. Suppose that the coefficient of $z$ on the right-hand side of \eqref{defprodfree} is nonzero. Then at least one summand $r_ws_y$ is nonzero, and for this summand $r_w\not=0$ and $s_y\not=0$, which by definition of the $\FF_R(w(X))_g$ imply $d(w)=g$ and $d(y)=h$. But now $d(z)=d(wy)=d(w)+d(y)=g+h$, so the product is in $\FF_R(w(X))_{g+h}$.
\end{proof}

\section{Kumjian-Pask families}\label{sec:KPfam}

The algebras of interest to us are algebraic analogues of a family of $C^*$-algebras introduced by Kumjian and Pask in \cite{KP}. In the algebraic analogue, the generating relations look a little different, so we begin by examining algebraic consequences of the relations in \cite{KP}. For the benefit of algebraists, we recall that a \emph{projection} in a $C^*$-algebra $A$ is an element $P\in A$ such that $P^*=P=P^2$. A \emph{partial isometry} is an element $S\in A$ such that $S=SS^*S$; equivalently, one of $SS^*$ or $S^*S$ is a projection, and then both are (see the appendix in \cite{R}, for example).

Let $\L$ be a row-finite $k$-graph without sources. Kumjian and Pask studied collections  $S
=\{S_{\lambda} : \lambda\in\L\}$ of partial isometries in a $C^*$-algebra $A$ such that
\begin{enumerate}
\item $\{S_v:v\in\L^0\}$ is a collection of mutually orthogonal projections, 
\item $S_{\lambda}S_{\mu} = S_{\lambda\mu}$ for $\lambda, \mu\in\L$ with $r(\mu) = s(\lambda)$,
\item\label{ck3} $S_{\lambda}^*S_{\lambda} = S_{s(\lambda)}$ for $\lambda\in\L$, and
\item\label{ck4} 
 $S_v = \sum_{\lambda\in v\L^n} S_{\lambda}S_{\lambda}^*$ for $v\in\L^0$ and  $n\in {\mathbb N}^k$.
\end{enumerate}
Although they did not use this name, these quickly became known as \emph{Cuntz-Krieger $\L$-families}.

The relation \eqref{ck3} immediately implies that $S_\lambda=S_\l(S_{\lambda}^*S_{\lambda})=S_\l S_{s(\l)}$. Next, recall that a finite sum $P=\sum_i P_i$ of projections in a $C^*$-algebra is a projection if and only if $P_iP_j=0$ for $i\not= j$, and then $PP_i=P_i$ for all $i$ (see \cite[Corollary~A.3]{R}). Thus, since $S_v$ is a projection, relation \eqref{ck4} implies that if $\lambda,\mu\in v\Lambda^n$ and $\lambda\not= \mu$, then $(S_\l S_\l^*)(S_\mu S_\mu^*)=0$ and $S_{v}(S_\l S_\l^*)=(S_\l S_\l^*)$. In particular, we have $S_{r(\l)}S_{\l}=S_{r(\l)}(S_\l S_\l^*)S_{\l}=(S_\l S_\l^*)S_{\l}=S_{\l}$. Next, note that
\[
S_{\l}^*S_{\mu}=S_{\l}^*(S_\l S_\l^*)(S_\mu S_\mu^*)S_{\mu},
\]
and hence we have the following stronger version of relation \eqref{ck3}:
\begin{enumerate}
\item[(\ref{ck3}$'$)] if $\lambda,\mu \in v\Lambda^n$, then $S_{\lambda}^*S_{\mu} = \delta_{\l,\mu} S_{s(\lambda)}$.
\end{enumerate}
The arguments in the previous paragraph do not work in the purely algebraic setting, and, as was the case for directed graphs in \cite{AA1}, we have to add some extra relations. 

If $\L$ is a $k$-graph, we let $\L^{\not=0}:=\{\l\in\L:d(\l)\not=0\}$, and for each $\l\in\L^{\not=0}$ we introduce a \emph{ghost path} $\l^*$; for $v\in \L^0$, we define $v^*:=v$. We write $G(\Lambda)$ for the set of ghost paths, or $G(\L^{\not=0})$ if we wish to exclude vertices. We define $d$, $r$ and $s$ on $G(\L)$ by
\[
d(\lambda^*) = -d(\lambda), \quad  r(\lambda^*) = s(\lambda), \quad s(\lambda^*) = r(\lambda);
\]
we then define composition on $G(\L)$ by setting $\lambda^*\mu^*
= (\mu\lambda)^*$ for $\lambda, \mu\in\L^{\neq 0}$ with $r(\mu^*) = s(\lambda^*)$. The factorization property of $\Lambda$ induces a similar factorization property on $G(\L)$. 

\begin{definition}
Let $\L$ be a row-finite $k$-graph without sources and let $R$ be a commutative ring with $1$. A \emph{Kumjian-Pask $\L$-family} $(P,S)$
in an $R$-algebra $A$ consists of two functions $P:\L^0\to A$ and $S:\L^{\not=0}\cup G(\L^{\not=0})\to A$ such that:
\begin{enumerate}
\item[(KP1)] $\{P_v:v\in \L^0\}$ is a family of mutually orthogonal idempotents,
\item[(KP2)] for all $\lambda, \mu\in\L^{\neq 0}$ with $r(\mu) = s(\lambda)$, we have
\[
S_{\lambda}S_{\mu} = S_{\lambda\mu}, \; S_{\mu^*}S_{\lambda^*} = S_{(\lambda\mu)^*}, \;
 P_{r(\lambda)}S_{\lambda} = S_{\lambda} = S_{\lambda}P_{s(\lambda)}, \;
  P_{s(\lambda)}S_{\lambda^*} = S_{\lambda^*} = S_{\lambda^*}P_{r(\lambda)},
\]
\item[(KP3)] for all $\lambda, \mu \in\L^{\neq 0}$ with $d(\lambda) = d(\mu)$, we have
\[
S_{\lambda^*}S_{\mu} = \delta_{\lambda,\mu}P_{s(\lambda)},
\]
\item[(KP4)] for all $v\in\L^0$ and all $n\in {\mathbb N}^k\setminus \{0\}$, we have
\[
P_v = \sum_{\lambda\in v\L^n} S_{\lambda}S_{\lambda^*}.
\]
\end{enumerate}
\end{definition}

\begin{remarks}\label{1stconseq} 
\begin{enumerate}
\item We have been careful to distinguish the vertex idempotents because we wanted to emphasise that there is only one generator for each path of degree $0$, whereas there are two for each path of nonzero degree. However, it is convenient when writing formulas such as \eqref{spanning} below to allow $S_v:=P_v$ and $S_{v^*}:=P_v$, and we do this. 

\item With the conventions we have set up, the last two relations in (KP2) can be summarized as $P_{r(x)}S_x=S_x=S_xP_{s(x)}$ for all $x\in \Lambda\cup G(\Lambda)$. This observation will be useful in calculations.

\item  Relations (KP2) and (KP3) imply that 
\[
(S_{\lambda}S_{\lambda^*})(S_{\lambda}S_{\lambda^*})=S_{\lambda}(S_{\lambda^*}S_{\lambda})S_{\lambda^*}=S_{\lambda}P_{s(\lambda)}S_{\lambda^*}=S_{\lambda}S_{\lambda^*},
\]
and  (KP3) gives $(S_{\lambda}S_{\lambda^*})(S_{\mu}S_{\mu^*})=0$ when $d(\l)=d(\mu)$ and $\l\not=\mu$. Thus for each $n$, $\{S_\l S_\l^*:\l\in \L^n\}$ is a set of mutually orthogonal idempotents.
\end{enumerate}
\end{remarks}

The following analogue of \cite[Lemma~3.1]{KP} tells us how to simplify products $S_{\lambda^*}S_{\mu}$.  

\begin{lemma}\label{pathprods}
Suppose that $(P, S)$ is a Kumjian-Pask $\L$-family, and $\lambda, \mu\in\L$. Then for each $q\geq d(\lambda)\vee d(\mu)$, we have
\[
S_{\lambda^*}S_{\mu} = \sum_{d(\lambda\alpha)=q,\;\lambda\alpha = \mu\beta}S_{\alpha}S_{\beta^*}.
\]
\end{lemma}
\begin{proof}
By (KP2), we have $S_{\lambda^*}S_{\mu}= P_{s(\lambda)}S_{\lambda^*}S_{\mu}P_{s(\mu)}$, and then applying (KP4) at $v = s(\lambda)$  and at $v =
s(\mu)$ gives
\begin{equation}\label{expandusingKP4}
S_{\lambda^*}S_{\mu}
=\sum_{\alpha\in s(\lambda)\L^{q-d(\lambda)},\;\beta\in s(\mu)\L^{q-d(\mu)}}S_{\alpha}S_{\alpha^*}S_{\lambda^*}S_{\mu}S_{\beta}S_{\beta^*}.
\end{equation}
Since $d(\lambda\alpha) = q = d(\mu\beta)$, (KP2) and (KP3) give
\[
S_{\alpha}S_{\alpha^*}S_{\lambda^*}S_{\mu}S_{\beta}S_{\beta^*}  =
  S_{\alpha}S_{(\lambda\alpha)^*}S_{\mu\beta}S_{\beta^*}
  =  \begin{cases}
S_{\alpha}S_{\beta^*} & \text{if $\lambda\alpha = \mu\beta$}\\
0 & \text{otherwise},
\end{cases}
\]
and so the right-hand side of \eqref{expandusingKP4} collapses as required.
\end{proof}

\begin{theorem}\label{$KP_R$}
Let $\L$ be a row-finite $k$-graph without sources, and let $R$ be a commutative ring with $1$. Then there is an $R$-algebra $\KP_R(\L)$ generated by a Kumjian-Pask $\L$-family $(p, s)$ such that, whenever $(Q,T)$ is a Kumjian-Pask $\L$-family in an $R$-algebra $A$, there is a
unique $R$-algebra homomorphism $\pi_{Q,T}:\KP_R(\L)\to A$ such that
\begin{equation}\label{defpiqt}
\pi_{Q,T}(p_v) = Q_v, \quad \pi_{Q,T}(s_{\lambda}) = T_{\lambda}, \quad \pi_{Q,T}(s_{\mu^*}) =T_{\mu^*}
\end{equation}
for $v\in \L^0$ and $\lambda,\mu\in\L^{\neq 0}$. There is a $\Z^k$-grading on $\KP_R(\L)$ satisfying 
\begin{equation}\label{spanning}
\KP_R(\Lambda)_n=\lsp_R\big\{s_{\lambda}s_{\mu^*}:\l,\mu \in \L\text{ and }d(\l)-d(\mu)=n\big\},
\end{equation}
and we have $rp_v\neq 0$ for $v\in \L^0$ and $r\in R\setminus\{0\}$.
\end{theorem}

Standard arguments show that $(\KP_R(\Lambda),(p,s))$ is unique up to isomorphism, and we call $\KP_R(\L)$ \emph{the Kumjian-Pask algebra} of $\L$ and $(p,s)$ \emph{the universal Kumjian-Pask $\L$-family}.

\begin{notationnonumber}
We find it helpful to use the convention that lower-case letters signify that a Kumjian-Pask family $(p,s)$ has a universal property. 
\end{notationnonumber}

The proof of this theorem will occupy the rest of the section.

We begin by considering the free algebra $\FF_R(w(X))$ on $X:=\Lambda^0\cup\Lambda^{\not=0}\cup G(\Lambda^{\not= 0})$. Let $I$ be the ideal\label{idealI} of ${\FF_R}(w(X))$ generated by the union of the following sets:
\begin{itemize}
\item $\big\{vw - \delta_{v,w}v:v, w\in \L^0\big\}$;

\smallskip
\item $\big\{\lambda - \mu\nu,\ \lambda^* - \nu^*\mu^*:\lambda,\mu, \nu\in
\L^{\neq 0} \text{ and }\lambda = \mu\nu\big\}$\\
\mbox{\qquad\qquad}$\cup\big\{r(\lambda)\lambda - \lambda,\ \lambda - \lambda
s(\lambda),\ s(\lambda)\lambda^* - \lambda^*,\ \lambda^* - \lambda^*r(\lambda) :\lambda\in\L^{\neq 0}\big\}$;

\smallskip
\item  $\big\{\lambda^*\mu - \delta_{\lambda,\mu}s(\lambda):\lambda, \mu \in\L^{\neq 0} \text{ such that }d(\lambda) = d(\mu)\big\}$;

\smallskip
\item $\big\{v - \sum_{\lambda\in v\L^n}  \lambda\lambda^*:v\in\L^0,\ n\in {\mathbb N}^k\setminus\{0\}\big\}$.
\end{itemize}
We now define $\KP_R(\L):={\FF}_R(w(X))/I$. Let $q:{\FF}_R(w(X))\to {\FF}_R(w(X))/I$ be the quotient map. Then $\{p_v,s_\l,s_{\mu^*}\}:=\{q(v), q(\l), q(\mu^*)\}$ gives a generating Kumjian-Pask $\L$-family $(p,s)$ in $\KP_R(\L)$.

Now let $(Q,T)$ be a Kumjian-Pask $\L$-family in an $R$-algebra $A$. Define $f_{Q,T}:X\to A$ by $f(v) = Q_v$, $f(\lambda) = T_{\lambda}$ and $f(\mu^*) = T_{\mu^*}$, and the universal property of $\FF_R(w(X))$ described in Proposition~\ref{univfreealg} gives an
$R$-algebra homomorphism $\phi_f:\FF_R(w(X))\to A$ such that $\phi_f(v) = Q_v$,
$\phi_f(\lambda) = T_{\lambda}$ and $\phi_f(\mu^*) = T_{\mu^*}$. The Kumjian-Pask relations imply that $\phi_f$ vanishes on the ideal $I$, and therefore factors through an $R$-algebra
homomorphism $\pi_{Q,T}:\KP_R(\L)\to A$ satisfying \eqref{defpiqt}. Since the elements in $X$ generate $\FF_R(w(X))$ as an algebra, there is exactly one such homomorphism.

Applying Proposition~\ref{gradefree} to the degree map $d:X\to \N^k$ gives a $\Z^k$-grading of the free algebra $\FF_R(w(X))$, and every generator of $I$ lies in one of the subgroups $\FF_R(w(X))_n$ of homogeneous elements. Thus the ideal $I$ is graded, and the quotient $\KP_R(\L)=\FF_R(w(X))/I$ is graded by the subgroups $q(\FF_R(w(X))_n)$. The following lemma  identifies $q(\FF_R(w(X))_n)$ with the subgroup $\KP_R(\L)_n$ described in \eqref{spanning}.

\begin{lemma}
For every $w\in w(X)$, we have $q(w)\in \KP_R(\L)_{d(w)}$.
\end{lemma}

\begin{proof}
We will prove this by induction on $|w|$. For $|w|=0$ or $1$, the result is covered by the convention in Remark~\ref{1stconseq}(a) that we can view vertices as paths or ghost paths, and hence can add appropriate factors $s_v=p_v$ or $s_{v^*}=p_v$ without changing $q(w)$.

For $|w|=2$, there are four cases to consider: $w=\l\mu^*$, $w=\l^*\mu$, $w=\lambda\mu$, $w=\mu^*\l^*$. For the first, we have $q(w)=s_{\l}s_{\mu^*}$, and there is nothing to prove. For the second, we apply Lemma~\ref{pathprods}, and observe that $\l\alpha=\mu\beta$ implies $d(\alpha)-d(\beta)=d(\mu)-d(\l)=d(w)$. For the third, we notice that the result is trivial if $q(w)=0$, and if not, (KP2) gives $0\not=q(w)=s_\l p_{s(\l)}p_{r(\mu)}s_{\mu}$, which implies that $s(\l)=r(\mu)$ and that $s_{\l}s_{\mu}=s_{\l\mu}s_{s(\mu)^*}$ belongs to $\KP_R(\L)_{d(w)}$. A similar argument works in the fourth case.

Now suppose that $n\geq 2$  and $q(y)\in  \KP_R(\L)_{d(y)}$ for every word $y$ with $|y|\leq n$. Let $w$ be a word with $|w|=n+1$ and $q(w)\not=0$. If $w$ contains a subword $w_iw_{i+1}=\l\mu$, then inserting vertex idempotents shows that $s(\l)=r(\mu)$, so that $\l$ and $\mu$ are composable in $\L$. We now let $w'$ be the word obtained from $w$ by replacing $w_iw_{i+1}$ with the single path $\l\mu$, and then
\[
q(w)=s_{w_1}\cdots s_{w_{i-1}}s_{\l}s_{\mu}s_{w_{i+2}}\cdots s_{w_{n+1}}
=s_{w_1}\cdots s_{w_{i-1}}s_{\l\mu}s_{w_{i+2}}\cdots s_{w_{n+1}}=q(w').
\]
Since $|w'|=n$ and $d(w')=d(w)$, the inductive hypothesis implies that $q(w)\in \KP_R(\L)_{d(w)}$. A similar argument shows that $q(w)\in \KP_R(\L)_{d(w)}$ whenever $w$ contains a subword $w_iw_{i+1}=\l^*\mu^*$.

If $w$ contains no subword of the form $\l\mu$ or $\l^*\mu^*$, then it must consist of alternating real and ghost paths. In particular, remembering that $|w|=n+1\geq 3$, we see that either $w_1w_2$ or $w_2w_3$ has the form $\l^*\mu$. Now we can use Lemma~\ref{pathprods} to write $q(w)$ as a sum of terms $q(y^i)$ with $|y^i|=n+1$ and $d(y^i)=d(w)$. Each nonzero summand $q(y^i)$ contains a factor of the form $s_{\beta^*}s_{\gamma^*}$ or one of the form $s_{\delta}s_\alpha$, and the argument of the preceding paragraph shows that every $q(y^i)\in \KP_R(\L)_{d(w)}$. Thus so is their sum $q(w)$.
\end{proof}

It remains to prove that the elements $rp_v$ with $r\not=0$ are nonzero, and for this it suffices to produce a Kumjian-Pask $\Lambda$-family $(Q,T)$ in an $R$-algebra such that each $rQ_v$ is nonzero. We do this by modifying the construction in \cite[Proposition 2.11]{KP}. Let $\FF_R(\L^\infty)$ be the free module with basis the infinite path space\label{infinpathrep}. We next fix $v\in \L^0$ and $\l,\mu\in\L^{\not=0}$, and use the composition and factorization constructions of Lemma~\ref{lemma2infinitepaths} to define functions $f_v, f_\l, f_{\mu^*}:\L^\infty\to
\FF_R(\L^\infty)$ by
\begin{align*}
f_v(x)  & =\begin{cases}
x& \text{if }x(0)=v\\
0 & \text{otherwise;}
\end{cases}\\
f_\l(x) &=  \begin{cases} \l x& \text{if }x(0)=s(\l)\\
 0 & \text{otherwise; and}
 \end{cases}\\
f_{\mu^*}(x) & =\begin{cases}
x(d(\mu),\infty)& \text{if }x(0,d(\mu))=\mu\\
0 & \text{otherwise.}
\end{cases}
\end{align*}
The universal property of free modules now gives nonzero endomorphisms $Q_v$, $T_\l$, $T_{\mu^*}:\FF_R(\L^\infty)\to \FF_R(\L^\infty)$ extending
$f_v$, $f_\l$ and $f_{\mu^*}$. \label{infinitepathrep}

It is straightforward to check using Lemma~\ref{lemma2infinitepaths} that $(Q,T)$ is a Kumjian-Pask $\L$-family in $\End(\FF_R(\L^\infty))$. For example, to verify (KP3), suppose that $d(\l)=d(\mu)$ and $x\in \L^\infty$. Then
\begin{align*}
T_{\l^*}T_\mu(x)= & \begin{cases}
T_{\l^*}(\mu x)& \text{if } x(0)=s(\mu),\\
0 & \text{otherwise}
\end{cases} \\
= & \begin{cases}
(\mu x)(d(\l),\infty) \phantom{\overset{X_Z}{Y}} & \text{if } x(0)=s(\mu) \text{ and } (\mu x)(0,d(\l))=\l,\\
0 & \text{otherwise.}
\end{cases}
\end{align*}
Since $d(\l)=d(\mu)$, Lemma~\ref{lemma2infinitepaths} implies that $(\mu x)(0,d(\lambda))=(\mu x)(0,d(\mu))=\mu$ if $r(x)=s(\mu)$, so $T_{\l^*}T_\mu(x)$ vanishes for all $x$ unless $\l=\mu$, and then is $x$ if and only if  $r(x)=s(\mu)$. But this is exactly what $Q_{s(\mu)}$ does to $x$, and hence we have $T_{\l^*}T_\mu=Q_{s(\mu)}$.

Since $(Q,T)$ is a Kumjian-Pask $\Lambda$-family, there exists an
$R$-algebra homomorphism
$\pi_{Q,T}:\KP_R(\Lambda)\to\End(\FF_R(\Lambda^\infty))$ such that
$\pi_{Q,T}(p_v)=Q_v$, $\pi_{Q,T}(s_\lambda)=T_\lambda$ and $\pi_{Q,T}(s_{\mu^*})=T_\mu$.
Since  every vertex $v$ is the range of an infinite path, if $r\neq 0$
then $rQ_v\neq 0$. It follows that $rp_v\neq 0$ too, and this
completes the proof of Theorem~\ref{$KP_R$}.

We call the $R$-algebra homomorphism $\pi_{Q,T}:\KP_R(\Lambda)\to\End(\FF_R(\Lambda^\infty))$  constructed above the
\emph{infinite-path representation of $\KP_R(\Lambda)$}.

\section{The uniqueness theorems}\label{sectionuniqueness}

Let  $\L$ be a row-finite $k$-graph without sources. We write $(p,s)$ for the universal Kumjian-Pask $\L$-family in $\KP_R(\L)$.
In this section we prove graded-uniqueness and Cuntz-Krieger uniqueness theorems for $\KP_R(\L)$. 

\begin{theorem}[The graded-uniqueness theorem]\label{gradeduniquess}
Let $\L$ be a row-finite $k$-graph without sources, $R$  a commutative ring with $1$,  and  $A$  a ${\mathbb Z}^k$-graded ring. If $\pi:\KP_R(\L)\to A$ is a ${\mathbb Z}^k$-graded ring homomorphism such that $\pi(rp_v)\neq 0$ for
all $r\in R\setminus \{0\}$ and $v\in \L^0$, then $\pi$ is injective.
\end{theorem}

The next two lemmas are the first steps in the proofs of both uniqueness theorems.

\begin{lemma}\label{normalform-lem} Every nonzero $x\in \KP_R(\L)$ can be written as a sum $\sum_{(\alpha,\beta)\in F} r_{\alpha,\beta} s_\alpha s_{\beta^*}$ where $F$ is a finite subset of $\Lambda\times\Lambda$, $r_{\alpha,\beta}\in R\setminus\{0\}$ for all $(\alpha,\beta)\in F$, and all the $\beta$ have the same degree.  In this case we say $x$ is written in \emph{normal form}.
\end{lemma}

\begin{proof}
By Theorem~\ref{$KP_R$}, we can write $x$ as a finite sum $x=\sum_{(\sigma,\tau)\in G} r_{\sigma,\tau} s_{\sigma}
s_{\tau^*}$ with each  $r_{\sigma,\tau}\not=0$. Set $m=\bigvee_{(\sigma,\tau)\in F}d(\tau)$.   For each $(\sigma,\tau)\in G$, applying (KP4) with $n_\tau:=m-d(\tau)$ gives
\[
s_{\sigma}s_{\tau^*}=s_{\sigma}p_{s(\sigma)}s_{\tau^*}=\sum_{\lambda\in s(\sigma)\L^{n_\tau}} s_{\sigma\lambda}s_{(\tau\lambda)^*};
\]
substituting back into the expression for $x$ and combining terms  gives the result.
\end{proof}

\begin{lemma}\label{key}
Suppose that $x$ is a nonzero element of $\KP_R(\L)$ and $x=\sum_{(\alpha,\beta)\in F} r_{\alpha,\beta} s_\alpha s_{\beta^*}$ is in normal form.   Then there exists $\gamma\in F_2:=\{\beta: (\alpha,\beta)\in F \text{\ for some\ }\alpha\in\L\}$ such that
\begin{equation}\label{key1}
0\neq xs_\gamma=\sum_{\alpha\in G} r_{\alpha,\gamma} s_\alpha\quad\text{where\ } G:=\{\alpha:(\alpha,\gamma)\in F\}.
\end{equation}
Further, if $\delta\in G$ then
\begin{equation}\label{key2}
0\neq s_{\delta^*} x s_\gamma=r_{\delta,\gamma}p_{s(\delta)}+\sum_{\{\alpha\in G\;:\;\alpha\neq\delta\}} r_{\alpha,\gamma}s_{\delta^*}s_\alpha,
\end{equation}
and $r_{\delta,\gamma}p_{s(\delta)}$ is the $0$-graded component of $s_{\delta^*} x s_\gamma$.
\end{lemma}

\begin{proof}
Since all $\beta$ in $F_2$ have the same degree,  (KP3) implies that $\{s_\beta s_{\beta^*}:\beta\in F_2\}$ is a set of mutually orthogonal idempotents.
Then $p=\sum_{\beta\in F_2}s_\beta s_{\beta^*}$ is an idempotent and satisfies $xp=x$. In particular, $xp\neq 0$, and hence there exists $\gamma\in
F_2$ such that $xs_\gamma\neq 0$. Now (KP3) gives
\begin{equation*}\label{followagain}
0\neq xs_\gamma=\sum_{(\alpha,\beta)\in F} r_{\alpha,\beta} s_\alpha s_{\beta^*}s_\gamma
= \sum_{\{(\alpha,\beta)\in F\;:\;\beta=\gamma
\}} r_{\alpha,\beta} s_\alpha=\sum_{\alpha\in G}r_{\alpha,\gamma} s_\alpha,
\end{equation*}
and for $\delta\in G$, we have
\begin{equation*}
s_{\delta^*}xs_\gamma=\sum_{\alpha\in G}r_{\alpha,\gamma}
 s_{\delta^*}s_\alpha=r_{\delta,\gamma} p_{s(\delta)}+\sum_{\{\alpha\in G\;:\;\alpha\neq \delta\}} r_{\alpha,\gamma}
s_{\delta^*}s_\alpha.
\end{equation*}
If $s_{\delta^*}s_\alpha\neq 0$ and $\alpha\neq\delta$, then $d(\alpha)\neq d(\delta)$ by (KP3), and $s_{\delta^*}s_\alpha$ is a sum of monomials $s_\mu s_{\nu^*}$ all of which have degree $d(\mu)-d(\nu)=d(\alpha)-d(\beta)\neq 0$ (see Lemma~\ref{pathprods}). 
Thus $r_{\delta,\gamma} p_{s(\delta)}$ is the $0$-graded
component of $s_{\delta^*}xs_\gamma$. 
Since $r_{\delta,\gamma} p_{s(\delta)}\neq 0$, we have $s_{\delta^*}xs_\gamma\neq 0$ too.\qedhere
\end{proof}

\begin{proof}[Proof of Theorem~\ref{gradeduniquess}]
Let $0\neq x\in \KP_R(\L)$. By Lemma~\ref{normalform-lem}, $x$ can be written in normal form, and by Lemma~\ref{key} there exist a finite set $G$ and $\gamma,\delta\in \L$ such that \eqref{key2} holds and  $r_{\delta,\gamma} p_{s(\delta)}$ is the $0$-graded
component of $s_{\delta^*}xs_\gamma$. Since $\pi$ is $\Z^k$-graded, $\pi(r_{\delta,\gamma} p_{s(\delta)})$ is the $0$-graded component of $\pi(s_{\delta^*}xs_\gamma)$, and since $\pi(r_{\delta,\gamma} p_{s(\delta)})$ is nonzero by assumption, so is $\pi(s_{\delta^*}xs_\gamma)$. Since $\pi$ is a ring homomorphism, we deduce that $\pi(x)\neq 0$, and hence that $\pi$ is injective.
\end{proof}

\begin{remark}
The graded-uniqueness theorem is an analogue of the gauge-invariant uniqueness theorems for graph $C^*$-algebras, and we will discuss the relationship in \S\ref{discussuongiut}. The first gauge-invariant uniqueness theorem was for Cuntz-Krieger algebras \cite[Theorem~2.3]{aHR}; the first versions for graph $C^*$-algebras  and higher-rank graph algebras were \cite[Theorem~2.1]{BPRS} and \cite[Theorem~3.4]{KP}. The graded-uniqueness theorem for Leavitt path algebras was originally derived from the classification of the graded ideals; direct proofs were given in \cite{RMalaga} and \cite{T}. Theorem~\ref{gradeduniquess} and its proof were motivated by \cite[Theorem~6.5]{T2}.
\end{remark}

For the  Cuntz-Krieger uniqueness theorem, we need an aperiodicity condition on $\Lambda$. Following Robertson and Sims \cite{RS}, we say
 that a $k$-graph $\L$ is  \emph{aperiodic}
if for every $v\in \L^0$ and $m\neq n\in {\mathbb N}^k$ there exists $\lambda\in v\L$ such that
$d(\lambda)\geq m\vee n$ and 
\begin{equation}\label{fpaperiodic}
\lambda(m,m+d(\lambda)-(m\vee n))\neq \lambda(n,n+d(\lambda)-(m\vee n)).
\end{equation} 
We say $\L$ is \emph{periodic} if $\L$ is not aperiodic.  Several aperiodicity conditions appear in the literature, but they are all equivalent when $\L$ is row-finite without sources.  We find the finite path formulation of aperiodicity from \cite{RS}  easier to understand,  and it allows us to  borrow arguments from  \cite{HRSW} which do not require readers to know about the different formulations in  \cite{KP} and \cite{RSY}.

\begin{example} Let $\Lambda$ be a row-finite $1$-graph without sources, and let $E=(E^0, E^1, r, s)$ be the associated directed graph.  
Then $\Lambda$ is aperiodic if and only if 
for every $v\in E^0$ and every $m,n\in {\mathbb N}$ with $m<n$, there exists $\lambda\in E^*$ with $r(\lambda)=v$, $|\lambda|\geq n$ and 
$\lambda_{m+1}\dots \lambda_{m+|\lambda|-n}\neq \lambda_{n+1}\dots\lambda_{|\lambda|}$.  \end{example}

The following reassuring lemma tells us that, for a directed graph, aperiodicity is equivalent to the usual hypothesis of Cuntz-Krieger uniqueness theorems.

\begin{lemma}\label{AperiodicityConditionL} Let $\L$ be a $1$-graph and $E$ its associated directed graph.
Then $\L$ is aperiodic if and only if every cycle in $E$ has an entry.
\end{lemma}
\begin{proof}
Suppose that $E$ has a cycle $\mu$ of length $k\geq 1$ without an entry, and take $v=r(\mu)$, $m=0$ and $n=k$.  Since $\mu$ has no entry, the only paths $\lambda$ with $r(\lambda)=r(\mu)$ and length at least $k$ have the form $\mu^l\mu'$, where $l\geq 1$ and $\mu=\mu'\mu''$; then  $\lambda_1\cdots\lambda_{|\lambda|-k}=\mu^{l-1}\mu'=\lambda_{k+1}\cdots\lambda_{|\lambda|}$ for every such $\lambda$, which shows that $\Lambda$ is periodic.

Conversely, suppose that every cycle in $E$ has an entry.   Fix $v\in E^0$ and $m<n$ in $\N$.  First, suppose that $v$ can be reached from a cycle $\mu$, that is, there exists $\alpha$ with $r(\alpha)=v$ such that $\alpha\mu$ is a path.  Then $\mu$ has an entry $e\in E^1$, and we may suppose by adjusting $\alpha$ that $s(\mu)=r(e)$.  Now choose a path of the form $\lambda=\alpha\mu\mu\dots\mu e$ such that $\lambda_m$ is an edge in $\mu$ and $|\lambda|\geq n$.  Then $\lambda_{m+|\lambda|-n}\neq\lambda_{|\lambda|}$.  Second, suppose that $v$ cannot be reached from a cycle. Choose $\lambda$ with $r(\lambda)=v$ and $|\lambda|>n$.  Then $\lambda_{m+1}\dots \lambda_{m+|\lambda|-n}\neq  \lambda_{n+1}\dots\lambda_{|\lambda|}$ because otherwise $\lambda_{m+1}\dots\lambda_n$ would be a return path which connects to $v$, and which would contain a cycle connecting to $v$. So either way, the aperiodicity condition holds for $m$, $n$ and $v$, and $\L$ is aperiodic.
\end{proof}

We can now state our second uniqueness theorem.

\begin{theorem}[The Cuntz-Krieger uniqueness theorem]\label{CKuniqueness}  Let $\L$ be an aperiodic row-finite $k$-graph without sources, let $R$ be a commutative ring with $1$, and let $A$ be a ring. If
$\pi:\KP_R(\L)\to A$ is a ring homomorphism such that $\pi(rp_v)\neq 0$ for all $r\in R\setminus\{0\}$ and $v\in \L^0$, then $\pi$ is injective.
\end{theorem}

We need two preliminary results for the proof. Lemma~\ref{aperiodicitylemma} was an ingredient in the proof of the $C^*$-algebraic uniqueness theorem in \cite{HRSW}, and Proposition~\ref{idealvertex} will be needed again in our analysis of the ideal structure in \S\ref{sectionsimple}.

\begin{lemma}\label{aperiodicitylemma}\textup{(\cite[Lemma~6.2]{HRSW})} Suppose that $\L$  an aperiodic row-finite $k$-graph without sources, and fix $v\in \L^0$ and $m\in {\mathbb N}^k$.  Then there exists $\lambda\in \L$ with $r(\lambda)=v$ and $d(\lambda)\geq m$ such that 
\begin{equation}\label{Robbie'sLemma}\left.\begin{array}{l} \alpha, \beta\in \L,\ s(\alpha)=s(\beta)=v,\
 d(\alpha),d(\beta)\leq m, \\
\text{and }(\alpha\lambda)(0,d(\lambda))=(\beta\lambda)(0,d(\lambda))\end{array}\right\}\Longrightarrow \alpha=\beta.
\end{equation}
\end{lemma}

\begin{proposition}\label{idealvertex} Let $\L$ be an aperiodic row-finite  $k$-graph without sources and let $R$ be a commutative ring with $1$. Let $x=\sum_{(\alpha,\beta)\in F}r_{\alpha,\beta}s_{\alpha}s_{\beta^*}$ be a nonzero element of  $\KP_R(\L)$ in normal form.
Then there exist $\sigma,\tau\in \L$, $(\delta,\gamma)\in F$ and $w\in \L^0$ such that $s_{\sigma^*}x s_\tau=r_{\delta,\gamma}p_w$.
\end{proposition}

\begin{proof}
Lemma~\ref{key} implies that there exists $\gamma\in \Lambda$ such that $G:=\{\alpha:(\alpha,\gamma)\in F\}$ is nonempty and
\begin{equation*}
0\neq
s_{\delta^*}xs_\gamma=r_{\delta,\gamma} p_{s(\delta)}+\sum_{\{\alpha\in G\;:\;\alpha\neq \delta\}} r_{\alpha,\gamma}
s_{\delta^*}s_\alpha\quad\text{for every $\delta\in G$.}
\end{equation*}
Since $\L$ is aperiodic we can apply Lemma~\ref{aperiodicitylemma} with $v=s(\delta)$ and $m=\bigvee_{\alpha\in G}d(\alpha)$ to find
$\lambda\in s(\delta)\L$ with $d(\lambda)\geq m$ such that  \eqref{Robbie'sLemma} holds. Now 
\begin{equation}
s_{\lambda^*}(s_{\delta^*}xs_\gamma)s_\lambda=r_{\delta,\gamma}
p_{s(\lambda)}+\sum_{\{\alpha\in G\;:\;\alpha\neq \delta\}} r_{\alpha,\gamma}
s_{(\delta\lambda)^*}s_{\alpha\lambda}.\label{eq-ac}
\end{equation}
If the summand $s_{(\delta\lambda)^*}s_{\alpha\lambda}$ is nonzero, then $s_{(\delta\lambda)(0,d(\lambda))^*}s_{(\alpha\lambda)(0,d(\lambda))}$ is nonzero, (KP3) implies
that $(\delta\lambda)(0,d(\lambda))=(\alpha\lambda)(0,d(\lambda))$, and \eqref{Robbie'sLemma} implies that $\alpha=\delta$. Thus \eqref{eq-ac} collapses to $s_{(\delta\lambda)^*}xs_{\gamma\lambda} =r_{\delta,\gamma} p_{s(\lambda)}$, and we can take $\sigma=\delta\lambda$ and $\tau=\gamma\lambda$.\end{proof}

\begin{proof}[Proof of Theorem~\ref{CKuniqueness}]
Let $0\neq x\in \KP_R(\L)$. By Lemma~\ref{normalform-lem} we can write   $x$  in normal
form. By Proposition~\ref{idealvertex} there exist $\sigma,\tau\in \L$ and $r\in R\setminus\{0\}$ such that $s_{\sigma^*}x s_\tau=rp_w$ for some $w\in \L^0$. Now  \[
\pi(s_{\sigma^*})\pi(x)\pi(s_\tau)=\pi(s_{\sigma^*}x s_\tau)=
\pi(rp_w)\neq 0
\] 
by assumption, and $\pi(x)\neq 0$. Thus $\pi$ is injective.
\end{proof}

The Cuntz-Krieger uniqueness theorem immediately gives:

\begin{cor}\label{CKimplypathrep}
Let $\Lambda$ be an aperiodic row-finite $k$-graph without sources. Then the infinite-path representation $\pi_{Q,T}:\KP_R(\Lambda)\to \End(\FF_R(\Lambda^\infty))$ from page~\pageref{infinpathrep} is injective.
\end{cor}

We will see in Lemma~\ref{kernelofpathrep} below that $\pi_{Q,T}$ is not injective when $\Lambda$ is periodic.

\begin{remark}
The uniqueness theorem for Cuntz-Krieger algebras was proved in \cite{CK}, and extended to graph algebras in \cite{KPR} and higher-rank graph algebras in \cite{KP}. The first versions for Leavitt algebras were in \cite{AA1, RMalaga, T}. All require some form of aperiodicity condition. For graphs, everybody now uses the condition~(L) from \cite{KPR}, which says that every cycle has an entry. For row-finite higher-rank graphs without sources, all the formulations are equivalent to the finite-path formulation which we use here \cite[Lemma~3.2]{RS}. When there are sources or infinite receivers, one has to be a bit more careful, and we refer to \cite{LS} for a detailed discussion. 
\end{remark}

\section{Basic ideals and basic simplicity}\label{sectionbasicsimplicity}

Let $\L$ be a row-finite $k$-graph without sources; we continue to write $(p,s)$ for the universal Kumjian-Pask $\L$-family in $\KP_R(\L)$.

A subset $H$ of $\L^0$ is \emph{hereditary} if $\lambda\in\L$ and $r(\lambda)\in H$ imply $s(\lambda)\in H$. A subset $H$ is \emph{saturated} if $v\in \L^0$, $n\in \N^k$ and $s(v\L^n)\subset H $ imply $v\in H$. For a saturated hereditary subset $H$, we write $I_H$ for the ideal of $\KP_R(\L)$ generated by $\{p_v:v\in H\}$.

The standard path for studying graph algebras predicts that $H\mapsto I_H$ should be a bijection between the saturated hereditary subsets of $\L^0$ and the graded ideals of $\KP_R(\L)$. However, since we are allowing coefficients in a commutative ring, we have to follow \cite{T2} and restrict attention to the \emph{basic ideals}, which are the ideals $I$ such that $rp_v\in I$ and $r\in R\setminus\{0\}$ imply $p_v\in I$. This assumption gets us back on path:

\begin{theorem} \label{latticehersat} Let $\L$ be a row-finite $k$-graph without sources and let $R$ be a commutative ring with $1$. Then the map $H\mapsto I_H$ is a lattice isomorphism from the lattice of saturated hereditary subsets of $\L^0$ onto the lattice of basic and graded ideals of $\KP_R(\L)$.
\end{theorem}

The proof of Theorem~\ref{latticehersat} follows the general path first taken in \cite[\S4]{BPRS}. The first lemma is a little more general than we need right now, but the sets $H_{I,r}$ will be of interest in \S\ref{sectionsimple}.

\begin{lemma}\label{hersatboth} Let $I$ be an ideal of $\KP_R(\L)$ and $r\in R$. Then $H_{I,r}:=\{v\in\L^0:rp_v\in I\}$  is a saturated hereditary subset of $\L^0$.  In particular, $H_I:=H_{I,1}=\{v\in\L^0:p_v\in I\}$ is saturated and hereditary.
\end{lemma}

\begin{proof} To see that $H_{I,r}$ is hereditary, suppose $\lambda \in \L$ and $r(\lambda)\in H_{I,r}$. Then  $rp_{r(\lambda)}\in I$ and $rs_\lambda=rp_{r(\lambda)}s_\lambda\in I$.  Now 
$rp_{s(\lambda)}=rs_{\lambda^*}s_{\lambda} = s_{\lambda^*}r s_{\lambda}\in I$. Thus $s(\lambda)\in H_{I,r}$, and $H_{I,r}$ is hereditary. To see that $H_{I,r}$ is saturated, fix $v\in \L^0$ and $n\in \N^k$, and suppose that 
$s(\lambda)\in H_{I,r}$ for all $\lambda\in v\L^n$.  Then $rp_{s(\lambda)}\in I$ for all $\lambda\in v\L^n$, and (KP4) gives
\[rp_v=\sum_{\lambda\in v\L^n }rs_\lambda s_{\lambda^*}=
\sum_{\lambda\in v\L^n }s_\lambda (rp_{s(\lambda)})s_{\lambda^*}\in I.
\] 
Thus $v\in H_{I,r}$, and $H_{I,r}$ is saturated.
\end{proof}

\begin{lemma} \label{quotientgraph} Suppose $\Lambda$ is a row-finite $k$-graph without sources and $H$ is a saturated hereditary subset of $\L^0$. Then  $\L\setminus H:=(\L^0\setminus H, s^{-1}(\L^0\setminus H),  r, s)$
is a row-finite $k$-graph without sources, and if $(Q,T)$ is a Kumjian-Pask family for $\L\setminus H$ in an $R$-algebra $A$, then
\[
P_v=\begin{cases}
Q_v & \text{if }v\not\in H\\
0 & \text{otherwise,}
\end{cases} \quad
S_\l=\begin{cases}
T_\l & \text{if }s(\l)\not\in H\\
0 & \text{otherwise,}
\end{cases} \quad\text{ and}\quad
S_{\mu^*}= \begin{cases}
T_{\mu^*}& \text{if }s(\mu)\not\in H\\
0& \text{otherwise}
\end{cases}
\]
form a Kumjian-Pask $\L$-family $(P,S)$ in $A$.
\end{lemma}

\begin{proof}
It is straightforward to check that $\L\setminus H$ is a subcategory of $\L$, and the hereditariness of $H$ implies that if $\lambda\in\L\setminus H$ and $\lambda=\mu\nu$, then the factors $\mu$ and $\nu$ have source in $\L^0\setminus H$ (see \cite[Theorem 5.2(b)]{RSY}). So $\Lambda\setminus H$ is a row-finite $k$-graph. To see that  $\Lambda\setminus H$ has no sources, suppose that $v\in (\L\setminus H)^0=\Lambda^0\setminus H$ and $n\in \N^k$. Since $\L$ has no sources, $v\L^n$ is nonempty, and  if $s(\lambda)\in H$ for every $\l\in v\L^n$, then $v\in H$ because $H$ is saturated, which contradicts $v\in\Lambda^0\setminus H$. Thus there must exist $\l\in v\L^n$ such that $s(\l)\in \Lambda^0\setminus H$, and then $\l\in v(\Lambda\setminus H)^n$, so $v$ is not a source in $\Lambda\setminus H$.

Most of the Kumjian-Pask relations (KP1--3) for $(P,S)$ follow immediately from those for $(Q,T)$, though we have to use that $H$ is hereditary to see that $s(\l)\notin H$ implies $r(\l)\notin H$, so that $S_\l=T_\l=Q_{r(\l)}T_\l=P_{r(\l)}S_\l$ in (KP2). For (KP4), we observe that the nonzero terms in $\sum_{\l\in v\L^n} S_\l S_{\l^*}$ are parametrized by
\[
\{\l\in v\L^n :s(\l)\not\in H \}= \begin{cases}
 \emptyset & \text{if }v\in H\\
v(\L\setminus H)^n& \text{if }v\notin H.\qedhere 
\end{cases} 
\]
\end{proof}

Recall that an ideal $I$ is \emph{idempotent} if
$I=I^2$ in the sense that $I$ is spanned by products $ab$ with $a,b\in I$.

\begin{lemma}\label{manyconditions} Let $H$ be a saturated hereditary subset of $\L^0$. Then 
\begin{equation}\label{formforIH}
I_H=\lsp_R\{s_\sigma s_{\l^*}:s(\sigma)=s(\l)\in H\},
\end{equation}
$I_H$ is a basic, graded and idempotent ideal of $\KP_R(\L)$, and 
$H_{I_H}=H$.
\end{lemma}
\begin{proof}
Since $s_\sigma
s_{\l^*}=s_\sigma p_{s(\sigma)}s_{\l^*}$, the  right-hand side $J$ of \eqref{formforIH} is contained in $I_H$, and it contains all the generators $p_v$ (by the convention in Remark~\ref{1stconseq}). So to prove \eqref{formforIH}, it suffices for us to prove that $J$ is an ideal.
To see this, consider $s_\sigma s_{\l^*}$ with $s(\sigma)=s(\l)\in H$ and $s_\mu s_{\delta^*}\in \KP_R(\L)$. Applying Lemma \ref{pathprods} with $q=d(\l)\vee d(\mu)$ gives
\begin{equation}\label{monomial}
s_\sigma s_{\lambda^*}s_\mu s_{\delta^*}
 = \sum_{\{\alpha\in \L^{q-d(\lambda)},\;\beta\in\L^{q-d(\mu)}\;:\;\lambda\alpha = \mu\beta\}}
 s_{\sigma\alpha}s_{(\delta\beta)^*}.
\end{equation}
Since $H$ is hereditary, $r(\alpha)=s(\sigma)$ and $r(\beta)=s(\l)$ imply that $s(\alpha)$ and $s(\beta)$ are in $H$. Thus each nonzero summand in 
\eqref{monomial} belongs to $J$.  Similarly,  $s_\mu s_{\delta^*}s_\sigma s_{\lambda^*}\in J$. Thus $J$ is an ideal, and we have proved \eqref{formforIH}.

To see that $I_H$ is idempotent, we suppose that $s(\sigma)=s(\lambda)\in H$, and observe that the spanning element $s_\sigma s_{\lambda^*}=(s_\sigma p_{s(\sigma)})(p_{s(\sigma)}s_{\lambda^*})$ for $I_H$ belongs to $(I_H)^2$. Since \eqref{formforIH} shows that $I_H$ is spanned by homogeneous elements, $I_H$ is graded.

To see that $I_H$ is basic and that $H=H_{I_H}$, it suffices to fix $r\not=0$ in $R$, and  prove that $v\notin H$ implies $rp_v\notin I_H$. Now consider the universal Kumjian-Pask $(\L\setminus H)$-family $(q,t)$ in $\KP_R(\L\setminus H)$, and extend it to a Kumjian-Pask $\L$-family $(P,S)$ as in Lemma~\ref{quotientgraph}. The universal property of $\KP_R(\L)$ (see Theorem~\ref{$KP_R$}) gives a homomorphism $\pi:=\pi_{P,S}:\KP_R(\L)\to \KP_R(\L\setminus H)$. Since $\pi(p_w)=0$ for $w\in H$, $\pi$ vanishes on $I_H$. On the other hand, applying Theorem~\ref{$KP_R$} to $\L\setminus H$ shows that $\pi(rp_v)=rq_v\not=0$ for every $v\in \L^0\setminus H$. Thus $rp_v$ cannot be in $I_H\subset \ker \pi$.
\end{proof}

\begin{proposition}\label{quotientKP} Let $\Lambda$ be a row-finite $k$-graph without sources and $R$ a commutative ring with $1$. Let $I$ be a basic ideal of $\KP_R(\L)$, and let $(q,t)$ and $(p,s)$ be the  universal Kumjian-Pask families in $KP_R(\L\setminus H_I)$ and $\KP_R(\L)$, respectively. If $I$ is graded or $\L\setminus H_I$ is aperiodic, then there exists an isomorphism $\pi:\KP_R(\L\setminus H_I)\to \KP_R(\L)/I$ such
that
\begin{equation}\label{pi}
\pi(q_v)=p_v+I\; ,\; \pi(t_\l)=s_\l+I\text{ and }\; \pi(t_{\mu^*})=s_{\mu^*}+I
\end{equation}
for $v\in \L^0\setminus H_I$ and $\l,\mu\in \L^{\neq 0}\cap s^{-1}(\Lambda^0\setminus H_I)$.
\end{proposition}
\begin{proof}
Observe that $\{p_v+I,s_\l+I,s_{\mu^*}+I\}$ 
is a Kumjian-Pask $(\L\setminus H_I)$-family $(p+I,s+I)$, and the universal property of $\KP_R(\L\setminus H_I)$ (Theorem~\ref{$KP_R$}) gives a homomorphism $\pi:=\pi_{p+I,s+I}$ satisfying \eqref{pi}. Since the other generators of $\KP_R(\Lambda)$ belong to $I$, the family $(p+I,s+I)$ generates $\KP_R(\L)/I$, and $\pi$ is surjective.

Suppose that $\pi(rq_v)=0$ for some $r\in R\setminus\{0\}$ and $v\not\in H_I$. Then
$rp_v+I=\pi(rq_v)=0$, so that $rp_v\in I$ and, since $I$ is basic, $p_v\in I$ as well. But this implies that
$v\in H_I$, a contradiction. Thus $\pi(rq_v)\neq 0$ for all $r\in R\setminus\{0\}$ and $v\not\in H_I$. If $\Lambda\setminus H_I$ is aperiodic, then the Cuntz-Krieger uniqueness theorem implies that $\pi$ is injective.

If $I$ is graded, then $\KP_R(\L)/I$ is graded by $(\KP_R(\L)/I)_n=q(\KP_R(\L)_n)$, where $q: \KP_R(\L)\to \KP_R(\L)/I$ is the quotient map. If $\alpha, \beta\in (\L\setminus
H_I)$ with $d(\alpha)-d(\beta)=n\in \Z^k$, then
\[\pi(t_\alpha t_{\beta^*})=s_\alpha s_{\beta^*}+I=q(s_\alpha s_{\beta^*})\in q(\KP_R(\L)_n)=(\KP_R(\L)/I)_n.\]
Thus $\pi$ is graded, and the graded-uniqueness theorem implies that $\pi$ is injective.  
\end{proof}

\begin{proof}[Proof of Theorem~\ref{latticehersat}] To see that $H\mapsto I_H$ is surjective, 
let $I$ be a  basic graded ideal. Then $H_I=\{v\in \L^0 : p_v\in I\}$ is
saturated  and hereditary by Lemma~\ref{hersatboth}. We will show that $I=I_{H_I}$. Since all the generators of $I_{H_I}$ lie in $I$, we have $I_{H_I}\subset I$. Consider the quotient map $Q:\KP_R(\L)/I_{H_I}\to \KP_R(\L)/I$. Since $H_{I_H}=H$ by Lemma \ref{manyconditions}, Proposition~\ref{quotientKP} gives us an isomorphism $\pi:\KP_R(\L\setminus H_I)\to \KP_R(\L)/I_{H_I}$. Now suppose $v$ belongs to $\L^0\setminus H_I$ and $r\not=0$. The composition $Q\circ \pi$ satisfies $Q\circ \pi(rp_v)=rp_v+I$, and since $I$ is basic
\[
Q\circ \pi(rp_v)=0\Longrightarrow  rp_v\in I\Longrightarrow p_v\in I\Longrightarrow v\in H_I
\]
which contradicts the choice of $v$. So $Q\circ \pi(rp_v)\not=0$, and it follows from the graded-uniqueness theorem (Theorem~\ref{gradeduniquess}) that $Q\circ\pi$ is injective. Thus $Q$ is injective, and $I=I_{H_I}$.

Injectivity of $H\mapsto I_H$ follows from Lemma~\ref{manyconditions}. Finally, since $H\subset K$ if and only if $I_H\subset I_K$, the map $H\mapsto I_H$ preserves least upper bounds and greatest lower bounds, and hence is a lattice isomorphism.
\end{proof}

The hypothesis that ``every $\L\setminus H$ is aperiodic'' in the next theorem is the analogue for $k$-graphs of Condition (K) for directed graphs.

\begin{theorem}\label{allbasicgraded}
Let $\L$ be a row-finite $k$-graph without sources and let $R$ be a commutative ring with $1$.  Then every basic ideal of $\KP_R(\L)$ is graded if and only if $\L\setminus H$ is aperiodic for every saturated hereditary subset $H$ of $\L^0$.
\end{theorem}

Theorem~\ref{allbasicgraded} and Theorem~\ref{latticehersat} together have the following corollary.

\begin{cor} Let $\L$ be a row-finite $k$-graph without sources and let $R$ be a commutative ring with $1$. Suppose that  $\L\setminus H$ is aperiodic for every saturated hereditary subset $H$ of $\L^0$. Then $H\mapsto I_H$ is an isomorphism of the lattice of saturated hereditary subsets of $\L^0$ onto the lattice of basic ideals in $\KP_R(\L)$.
\end{cor}

To prove Theorem~\ref{allbasicgraded} we need some more results.  The next lemma is \cite[Lemma~3.3]{RS}; since the proof in \cite{RS} invokes results about a different formulation of periodicity, we give a direct proof.

\begin{lemma}\label{Lambdaperiodic}
Suppose that $\Lambda$ is periodic. Then there exist $v\in\L^0$ and $m\neq n\in\N^k$ such that, for all $\mu\in v\L^m$ and $\alpha\in s(\mu)\L^{(m\vee n)-m}$, there exists $\nu\in v\Lambda^n$ with
$\mu\alpha y=\nu\alpha y$ for all $y\in s(\alpha)\L^\infty$.
\end{lemma}

\begin{proof}
Since  $\L$ is periodic, there exist $v\in\L^0$ and $m\neq n\in\N^k$ such that for all $\lambda\in v\Lambda$ with $d(\lambda)\geq m\vee n$ we have 
\begin{equation}\label{perlambda}
\lambda(m, m+d(\l)-(m\vee n))=\lambda(n, n+d(\l)-(m\vee n)).
\end{equation}  
For every $x\in v\Lambda^\infty$ and $l\in \N^k$, we can apply \eqref{perlambda} to $\lambda=x(0,(m\vee n)+l)$, and deduce that $x(m,m+l)=x(n,n+l)$; in other words, for all $x\in v\Lambda^\infty$, we have $x(m,\infty)=x(n,\infty)$.  Now take $\nu=(\mu\alpha)(0,n)$, and let $y\in s(\alpha)\Lambda^\infty$.  Then $x:=\mu\alpha y$ belongs to $v\Lambda^\infty$, and hence
\begin{align*}
\mu\alpha y&=(\mu\alpha y)(0,n)(\mu\alpha y)(n,\infty)
=(\mu\alpha)(0,n)(\mu\alpha y)(n,\infty)\\
&=\nu(\mu\alpha y)(n,\infty)
=\nu(\mu\alpha y)(m,\infty)=\nu\alpha y.\qedhere
\end{align*}
\end{proof}

The following lemma is used in the proofs of Proposition~\ref{aperiodicbasicideal} and Theorem~\ref{allbasicgraded}.

\begin{lemma}\label{kernelofpathrep}
Let $\pi_{Q,T}:\KP_R(\L)\to\End(\FF_R(\L^\infty))$ be the infinite-path representation constructed on page~\pageref{infinitepathrep}. If $\L$ is periodic then there exist $\mu,\nu,\alpha\in\L$ such that \[0\neq s_{\mu\alpha}s_{(\mu\alpha)^*}-s_{\nu\alpha}s_{(\mu\alpha)^*}\in\ker\pi_{Q,T}.\]
\end{lemma}

\begin{proof}
Take $v\in \L^0$, $m\neq n\in \N^k$ as given by Lemma~\ref{Lambdaperiodic}, and choose $\mu\in v\L^m$ and $\alpha\in s(\mu)\L^{(m\vee n)-m}$. Then there exists $\nu\in v\Lambda^n$ such that $\mu\alpha y=\nu\alpha y$ for all $y\in \L^\infty$.
Suppose, by way of contradiction, that $a:=s_{\mu\alpha}s_{(\mu\alpha)^*}-s_{\nu\alpha}s_{(\mu\alpha)^*}=0$. Then $s_{\mu\alpha}s_{(\mu\alpha)^*}=s_{\nu\alpha}s_{(\mu\alpha)^*}$.  But $d(s_{\mu\alpha}s_{(\mu\alpha)^*})=d(\mu\alpha)-d(\mu\alpha)=0$, whereas 
\[
d(s_{\nu\alpha}s_{(\mu\alpha)^*})=d(\nu\alpha)-d(\mu\alpha)=d(\nu)+d(\alpha)-d(\mu)-d(\alpha)=n-m\neq 0.  
\]
Thus $s_{\mu\alpha}s_{(\mu\alpha)^*}=s_{\nu\alpha}s_{(\mu\alpha)^*}=0$.
But now
$0=s_{(\mu\alpha)^*}(s_{\mu\alpha}s_{(\mu\alpha)^*})s_{\mu\alpha}=p_{s(\mu\alpha)}^2=p_{s(\alpha)}$
contradicts Theorem~\ref{$KP_R$}.  Hence $a\neq 0$.

To see that $a\in \ker \pi_{Q,T}$ we fix $x\in \L^\infty$ and show that $\pi(a)(x)=0$. Recall that
$\pi_{Q,T}(s_\l)=T_\l$ and $\pi_{Q,T}(s_{\l^*})=T_{\l^*}$ where
\[
T_\l(x) =  \begin{cases}
\l x & \text{if }x(0)=s(\l)\\
0 & \text{otherwise}
\end{cases}\quad \text{and}\quad
T_{\l^*}(x)  = \begin{cases}
x(d(\l),\infty)& \text{if }x(0,d(\l))=\l\\
0 & \text{otherwise.}
\end{cases}
\]
If $x(0,d(\mu\alpha))\neq \mu\alpha$ then $T_{(\mu\alpha)^*}(x)=0$ and hence
$\pi_{Q,T}(a)(x)=T_{\mu\alpha}T_{(\mu\alpha)^*}(x)-T_{\nu\alpha}T_{(\mu\alpha)^*}(x)=0$. On the other hand, if
$x(0,d(\mu\alpha))=\mu\alpha$, then 
$\pi_{Q,T}(a)(x)=(T_{\mu\alpha}-T_{\nu\alpha})(x(d(\mu\alpha),\infty))$ has the form $\mu\alpha y -\nu\alpha y$ for $y=x(d(\mu\alpha),\infty)$, and hence $\pi_{Q,T}(a)(x)=0$. Thus $a\in \ker\pi_{Q,T}$.
\end{proof}

\begin{cor}\label{periodic}
Suppose that $\Lambda$ is a row-finite $k$ graph without sources. Then
the infinite-path representation $\pi_{Q,T}$ from page~\pageref{infinitepathrep}  is injective if and only if $\Lambda$ is aperiodic.
\end{cor}

\begin{proof}
Lemma~\ref{kernelofpathrep} shows that $\ker \pi_{Q,T}$ is nonzero when $\Lambda$ is periodic, and the converse is Corollary~\ref{CKimplypathrep}.
\end{proof}

\begin{proposition}\label{aperiodicbasicideal} Let $\Lambda$ be a row-finite $k$-graph without sources, and let $R$ be a commutative ring with $1$. Then $\L$ is aperiodic if and only if every nonzero  basic ideal of $\KP_R(\L)$ contains a vertex idempotent $p_v$.
\end{proposition}

\begin{proof}
If $\L$ is periodic, then we know from Lemma~\ref{kernelofpathrep} that the kernel of the infinite-path representation is nonzero and basic, and by construction contains no $p_v$. So suppose that $\L$ is aperiodic, and $I$ is a basic ideal in $\KP_R(\L)$ such that $p_v\not\in I$ for all $v\in \L^0$; we want to show that $I=\{0\}$. 

 If either $s_\l\in I$ or $s_{\l^*}\in I$ then $p_{s(\l)}=s_{\l^*}s_\l\in
I$, contradicting the assumption. Thus $p_v+I,s_\l+I,s_{\mu^*}+I$ are  nonzero for all $v\in\L^0$ and $\lambda,\mu\in\Lambda^{\neq0}$, and they form a Kumjian-Pask $\L$-family in $\KP_R(\L)/I$ which induces a surjective homomorphism $\pi_{p+I,s+I}:\KP_R(\L)\to \KP_R(\L)/I$ such that $\pi_{p+I,s+I}(p_v)=p_v+I$. 

Suppose that $\pi_{p+I,s+I}(rp_v)=0$ for some $r\in R\setminus\{0\}$. Then $0=\pi_{p+I,s+I}(rp_v)=r(p_v+I)$
implies that $rp_v\in I$, and,  since $I$ is basic, this implies $p_v\in I$,  a contradiction.  Thus $\pi_{p+I,s+I}(rp_v)\neq 0$ for all $r\in R\setminus\{0\}$. Since $\Lambda$ is aperiodic,  the Cuntz-Krieger uniqueness theorem implies that $\pi_{p+I,s+I}$ is an isomorphism. But $\pi_{p+I,s+I}$ is the quotient map, and hence $I=\{0\}$, as required.
\end{proof}

\begin{proof}[Proof of Theorem~\ref{allbasicgraded}] Suppose that $\L^0$ contains a saturated hereditary subset $H$ such that $\L\setminus H$ is periodic. Let $(q,t)$ be the universal Kumjian-Pask $\L\setminus H$ family in $\KP_R(\L\setminus H)$. Then Lemma~\ref{kernelofpathrep} implies that the kernel of the infinite-path representation is a nonzero ideal in $\KP_R(\L\setminus H)$ which contains no $rq_v$, and pulling this ideal over under the isomorphism of Proposition~\ref{quotientKP} gives an ideal $K$ in $\KP_R(\L)/I_H$ which contains no $r(p_v+I_H)$ for $r\not=0$ and $v\notin H$. But then the inverse image of $K$ in $\KP_R(\L)$ is an ideal $J$ which strictly contains $I_H$ and satisfies
\begin{align*}
rp_v\in J\text{ for some }r\not=0&\Longrightarrow r(p_v+I_H)\in K\text{ for some }r\not=0\\
&\Longrightarrow p_v\in J\\
&\Longrightarrow v\in H.
\end{align*}
These implications show, first, that $J$ is basic, and, second, that $H_J=H$. But then $J\not=I_{H_J}=I_H$, and $J$ cannot be graded by Theorem~\ref{latticehersat}.

Conversely, suppose that every $\L\setminus H$ is aperiodic, and that $J$ is a nonzero basic ideal of $\KP_R(\L)$. We trivially have $I_{H_J}\subset J$, and we claim that in fact $I_{H_J}=J$. Suppose not. Then $J/I_{H_J}$ is a nonzero ideal in $\KP_R(\L)/I_{H_J}$, and its image $L$ under the isomorphism of Proposition~\ref{quotientKP} is a nonzero ideal in $\KP_R(\L\setminus H_J)$. This ideal $L$ is basic: if $r\not=0$ and $q_v$ is a vertex idempotent in $\KP_R(\L\setminus H_J)$, then
\[
rq_v\in L\Longrightarrow rp_v+I_{H_J}\in J/I_{H_J}
\Longrightarrow rp_v\in J\Longrightarrow p_v\in J\Longrightarrow q_v\in L.
\]
Since $\L\setminus H_J$ is aperiodic, Proposition~\ref{aperiodicbasicideal} implies that $L$ contains some $q_v$ for $v\in \L^0\setminus H_J$. But then $J$ contains $p_v$, and $v\in H_J$, which is a contradiction. Thus $J=I_{H_J}$, and Lemma~\ref{manyconditions} implies that $J$ is graded.
\end{proof}

As in \cite{T2}, we say that $\KP_R(\Lambda)$ is \emph{basically simple} if its only basic ideals are $\{0\}$ and $\KP_R(\Lambda)$. If $R$ is a field, then every ideal is basic, and hence basic simplicity is the same as simplicity.

Our next goal is necessary and sufficient conditions for basic simplicity of $\KP_R(\L)$.  We do this independently of Theorem~\ref{latticehersat} by following the approach of \cite{RS}.  
A $k$-graph $\L$ is \emph{cofinal} if for every $x\in \L^\infty$ and every $v\in \L^0$, there exists $n\in \N^k$ such that $v\L x(n)\neq \emptyset$.  This cofinality condition is based on the one used for directed graphs in \cite[\S3]{KPRR}.

\begin{lemma}\label{cofinalimplieshersat} If $\L$ is cofinal then the only saturated hereditary
subsets of $\Lambda^0$ are $\emptyset$ and $\L^0$.
\end{lemma}

\begin{proof}
Suppose there exists a nontrivial saturated hereditary subset $H$ of $\L^0$. Choose $v\in
\L^0\setminus H$ and $w\in H$. Choose a sequence $\{n(i)\}$ in $\N^k$ such that $n(i)\leq n(i+1)$ and $n(i)\to\infty$
in the sense that $n(i)_j\to \infty$ as $i\to \infty$ for $1\leq j\leq k$.
Since $v\notin H$ and $H$ is saturated, there exists  $\lambda_1\in v\Lambda^{n(1)}$ such that $s(\lambda_1)\notin H$.  By induction, for $i\geq 1$ there exists $\lambda_{i+1}\in s(\lambda_i)\Lambda^{n(i+1)-n(i)}$ such that $s(\lambda_{i+1})\notin H$. Now set $\mu_1=\lambda_1$ and $\mu_{i+1}=\mu_i\lambda_{i+1}$ for $i\geq 1$. Then $\mu_{i+1}(0, n(i))=\mu_i$, and by Lemma~\ref{lemma1infinitepaths}  there exists $y\in\L^\infty$ such that $y(0,n(i))=\mu_i=\lambda_1\dots\lambda_i$.

Since $\Lambda$ is cofinal,  there exists $m\in \N^k$ such that $w\L y(m)\neq \emptyset$. Since $w\in H$ and $H$ is
hereditary, we have $y(m)\in H$. Choose $i_0\in \N$
such that $n(i_0)\geq m$. Then $y(n(i_0))=s(\l_{i_0})$ belongs to $H$ because $H$ is hereditary.  But $s(\l_{i_0})\notin H$ by construction, and we have a contradiction.  So the only saturated hereditary subsets are the trivial ones.
\end{proof}

\begin{proposition}\label{cofinalbasicideal} Let $\Lambda$ be a row-finite $k$-graph without sources, and let $R$ be a commutative ring with $1$. Then $\L$ is cofinal if and only if the only basic ideal containing a vertex idempotent $p_v$ is $\KP_R(\L)$.
\end{proposition}

\begin{proof}
Suppose that $\L$ is cofinal, and $I$ is a basic ideal containing some $p_w$. Then $H_I=\{v\in \L^0:p_v\in I\}$ is  nonempty, and is saturated and hereditary by Lemma \ref{hersatboth}. Since $\Lambda$ is cofinal, $H_I=\L^0$ by Lemma~\ref{cofinalimplieshersat}. Thus $p_v\in I$ for all $v\in\Lambda^0$, and we have
\[ \KP_R(\L)=\lsp\{s_\alpha p_{s(\alpha)}s_{\beta^*} : \alpha,\beta\in \L^{\neq 0},
s(\alpha)=s(\beta)\}\subset I.
\]

Now suppose that $\L$ is not cofinal. Then there exist $v\in \L^0$ and an infinite path $x\in \L^\infty$ such that $v\L x(n)=\emptyset$ for every $n\in \N^k$.  By \cite[Proposition 3.4, proof of (ii) $\Rightarrow$ (i)]{RS} the set $H_x:=\{w\in \L^0 : w\L x(n)=\emptyset \text{ for all }n\in \N^k\}$ is a saturated hereditary subset of $\L^0$. Note that $H_x$ is nontrivial since $v\in H_x$ and $x(0)\notin H_x$.
Now $I_{H_x}$ is a basic ideal of $\KP_R(\Lambda)$ by Lemma~\ref{manyconditions}, and $p_v\in I_{H_x}$.  But ${H}_{I_{H_x}}=H_x$ by Lemma~\ref{manyconditions}, and hence $p_{x(0)}\notin I_{H_x}$ because $x(0)\notin {H_x}$. Thus $I_{H_x}\neq \KP_R(\Lambda)$, and we have a nontrivial ideal containing a vertex idempotent.
\end{proof}

\begin{theorem}\label{basicallysimple} Let $\L$ be a row-finite $k$-graph without sources, and let $R$ be a commutative ring with $1$.  Then  $\KP_R(\L)$ is basically simple if and
only if the graph $\L$ is cofinal and aperiodic.
\end{theorem}
\begin{proof}
If $\KP_R(\L)$ is basically simple, then the only nonzero basic ideal is $\KP_R(\L)$. So Proposition~\ref{aperiodicbasicideal} implies that $\L$ is aperiodic, and Proposition~\ref{cofinalbasicideal} that $\L$ is cofinal.

Conversely, assume that $\L$ is cofinal and aperiodic and $I$ is a nonzero basic
ideal in $\KP_R(\L)$. By Proposition~\ref{aperiodicbasicideal} there exists $v\in \L^0$ with $p_v\in I$.
But then  $I=\KP_R(\L)$ by Proposition~\ref{cofinalbasicideal}. Thus $\KP_R(\L)$ is basically
simple.
\end{proof}

\begin{remark}
The parametrization of ideals in Cuntz-Krieger algebras by the saturated hereditary subsets  comes from \cite{C}, and was extended to various classes of graph $C^*$-algebras in \cite{KPRR, BPRS, BHRS, HS}. The ideals in the $C^*$-algebras of higher-rank graphs were first analyzed in \cite{RSY}.  The graded ideals in the Leavitt path algebras were described in \cite{AMP}, \cite{T} and \cite{T2}. The simplicity theorem for $C^*$-algebras goes back to Cuntz and Krieger \cite{CK}, and for Leavitt path algebras to Abrams and Aranda Pino \cite{AA1}. Our proof of basic simplicity was inspired by the work of Robertson and Sims \cite{RS}.
\end{remark}


\section{Simplicity}\label{sectionsimple}

Let $\L$ be a row-finite $k$-graph without sources, and write $(p,s)$ for the universal Kumjian-Pask family in $\KP_R(\Lambda)$.  So far the ring $R$ has played little role in our study of $\KP_R(\L)$; in fact, the notion of basic ideal in the previous section was  engineered by Tomforde to avoid dealing with ideals in $R$. The main result of this section is:

\begin{theorem}\label{simplicity} Suppose that $\Lambda$ is a row-finite $k$-graph without sources, and that $R$ is a commutative ring with $1$.  Then $\KP_R(\L)$ is simple if and only if $R$ is a field and $\L$ is aperiodic and cofinal.
\end{theorem}

This theorem was motivated by the following observations.  If $R$ is an algebra over a commutative ring $S$, then \cite[Theorem 8.1]{T2} implies that $L_R(E)$ is isomorphic to $R \otimes_S L_S(E)$ as an $R$-algebra. Moreover, if $A$ is an $s$-unital algebra over a field $K$, and $E$ is  a cofinal graph in which every cycle has an entry,  then \cite[Corollary 7.8]{AGPS} implies that every ideal of
$A\otimes_K L_K(E)$ has the form $I\otimes_K L_K(E)$ for some ideal $I$ of $A$.

We write $\Ll(A)$ for the lattice of ideals of a ring $A$. Then we can define restriction and induction maps 
\[\Res:\Ll(\KP_R(\Lambda))\to\Ll(R)\quad\text{and}\quad \Ind:\Ll(R)\to \Ll(\KP_R(\L)),
\]
as follows:
\begin{gather*}\Res I:=\{r\in R : rp_v\in I \mbox{ for all }v\in \L^0\}\\
\Ind M:=\lsp_R\{r s_\alpha s_\beta^*: r\in M, \alpha,\beta\in \Lambda\}.
\end{gather*}
One can easily check that $\Res I$ and $\Ind M$ are ideals in $R$ and $\KP_R(\L)$, respectively. 

We will need the following lemma in Proposition~\ref{quotient} and in Proposition~\ref{res-ind}.

\begin{lemma}\label{helper}
Let $M$ be an ideal of $R$, $r\in R$, and $v\in \Lambda^0$. If $rp_v\in\Ind M$, then $r\in M$.
\end{lemma}
\begin{proof}
If $rp_v=0$ then $r=0$ and is in $M$. So suppose $rp_v\not=0$. 
We have $rp_v=\sum_{(\alpha,\beta)\in F}
r_{\alpha,\beta}s_{\alpha}s_{\beta^*}$ for some $r_{\alpha,\beta}\in M\setminus\{0\}$; by Lemma~\ref{normalform-lem} we may assume this is in normal form, and a glance at the proof of  Lemma~\ref{normalform-lem} shows that the $r_{\alpha,\beta}$ are then still in $M\setminus\{0\}$.
By Lemma~\ref{key}  there exists $\gamma\in \Lambda$ and a finite set $G\subset\L$ such that
$0\neq rp_vs_\gamma=\sum_{\alpha\in G} r_{\alpha,\gamma} s_\alpha$.
Since $\KP_R(\L)$ is ${\mathbb Z}^k$-graded  we have 
\[
 0\neq(rp_v)s_\gamma=\sum_{\{\alpha\in G\;:\;d(\alpha)=d(\gamma)\}}r_{\alpha,\gamma} s_\alpha.
\]
We must have $v=r(\gamma)$, and applying (KP3) gives
\begin{align*}
rp_{s(\gamma)}  =rs_{\gamma^*}s_\gamma&= s_{\gamma^*}(rp_v)s_\gamma=\sum_{\{\alpha\in G : d(\alpha)=
d(\gamma)\}}r_{\alpha,\gamma}  s_{\gamma^*}s_\alpha\\
&= \begin{cases}r_{\gamma,\gamma} p_{s(\gamma)} &
\text{if $\gamma\in G$}\\ 
0 &
\text{otherwise.}
\end{cases}
\end{align*}
But now   either $(r-r_{\gamma,\gamma})p_{s(\gamma)}=0$  or $rp_{s(\gamma)}=0$, and hence either $r=r_{\gamma,\gamma}$ or $r=0$ by Theorem~\ref{$KP_R$}. In either case, $r\in M$.
\end{proof}

\begin{proposition}\label{quotient} Suppose that $\Lambda$ is a row-finite $k$-graph without sources, that $R$ is a commutative ring with $1$ and  that $M$ is a proper ideal of $R$.  Then $\KP_R(\L)/\Ind M$ is an $R/M$-algebra with $(r+M)(x+\Ind M)=rx+\Ind M$, and there is an isomorphism $\pi$ of $\KP_{R/M}(\L)$ onto $\KP_R(\L)/\Ind M$ which takes the universal Kumjian-Pask family $(q,t)$ in $\KP_{R/M}(\Lambda)$ to $(p+\Ind M, s+\Ind M)$.
\end{proposition}
\begin{proof}
To see the action of $R/M$ is well-defined, note that if $r+M=s+M$
and $x+\Ind M=y+\Ind M$, then 
\[
rx-sy=r(x-y)+(r-s)y\in R \cdot \Ind M+M \cdot
\KP_R(\L)\subset \Ind M,
\] as required.

The set $(p+\Ind M,s+\Ind M)$ is a Kumjian-Pask family in $\KP_R(\L)/\Ind M$, and thus the universal
property of $\KP_{R/M}(\L)$ (Theorem~\ref{$KP_R$}) gives a homomorphism $\pi$ taking $(q,t)$ to $(p+\Ind M, s+\Ind M)$; $\pi$ is surjective because $(p,s)$ generates $\KP_R(\Lambda)$. The ideal  $\Ind M$ is spanned by homogeneous elements, and hence is graded;  then $\KP_R(\L)/\Ind M$ is graded by the images $q(\KP_R(\L)_n)$ under the quotient map $q$. The homomorphism $\pi$ is then a graded homomorphism. Since $M$ is proper, Lemma~\ref{helper} implies that no vertex projection $p_v$ belongs to $\Ind M$, and hence each vertex projection $p_v+\Ind M$ in the quotient is nonzero. Thus the graded-uniqueness theorem implies that $\pi$ is injective.
\end{proof}

\begin{proposition}\label{res-ind} Let $\Lambda$ be a row-finite $k$-graph without sources, and let $R$ be a commutative ring with $1$.
\begin{enumerate}
\item\label{res-ind1}   We have $\Res\circ \Ind=\id$. In particular, $\Ind$ is injective.
\item\label{res-ind2}   Suppose that $\L$ is aperiodic and cofinal.  Then   $\Ind\circ\Res=\id$, and $\Ind:\Ll(R)\to \Ll(\KP_R(\Lambda))$ is a lattice isomorphism with inverse $\Res$.
\end{enumerate}
\end{proposition}

\begin{proof}  \eqref{res-ind1}  Let $M$ be an ideal of $R$. We will show that $\Res\circ\Ind(M)=M$, and the
injectivity of $\Ind$ then follows. If $m\in M$ then $mp_v\in\Ind M$ for all $v\in\L^0$, and hence $m\in\Res\circ\Ind M$.   Thus $M\subset \Res\circ\Ind M$.  
For the reverse inclusion, let $t\in \Res\circ\Ind M$. Then $tp_v\in \Ind M$ for $v\in\L^0$ and hence $t\in M$ by Lemma~\ref{helper}.

\eqref{res-ind2} 
Let $I$ be a nonzero ideal of $\KP_R(\L)$. We will show that $\Ind\circ\Res I=I$, and the surjectivity of $\Ind$ then follows. Let   $0\neq x\in I$. We write $x$ in normal form
$\sum_{(\alpha,\beta)\in F} r_{\alpha,\beta}s_{\alpha}s_{\beta^*}$  (see Lemma~\ref{normalform-lem}). 
Since $\Lambda$ is aperiodic, 
by Proposition~\ref{idealvertex} there exist  $\sigma,\tau\in \L$ and
$(\delta,\gamma)\in F$ such that $s_{\sigma^*}x s_\tau=r_{\delta,\gamma}p_w$ for some $w\in \L^0$. 
Then $r_{\delta,\gamma} p_w\in I$, and thus $w$ is in the saturated hereditary subset  $H_{I, r_{\delta,\gamma}}$ of  Lemma~\ref{hersatboth}. Since $\Lambda$  is cofinal by hypothesis, Lemma~\ref{cofinalimplieshersat} implies that
$H_{I, r_{\delta,\gamma}}=\L^0$, so that $r_{\delta,\gamma}p_v\in I$ for all $v\in \L^0$.
In particular, $r_{\delta,\gamma}p_{r(\delta)}\in I$, and hence
\[y:=x-r_{\delta,\gamma}p_{r(\delta)}s_{\delta}s_{\gamma^*}=\sum_{(\alpha,\beta)\in F\setminus\{(\delta,\gamma)\}} r_{\alpha,\beta}s_{\alpha}s_{\beta^*}
\] 
belongs to $I$ and
is in normal form.  Repeating the above process $|F|-1$ times gives $r_{\alpha,\beta}p_v\in I$ for all $v\in \L^0$
and $(\alpha, \beta)\in F$.  Thus $r_{\alpha,\beta}\in \Res I$ for $(\alpha, \beta)\in F$, and hence $x\in\Ind\circ\Res I$.  Thus $I\subset \Ind\circ\Res I$.

For the reverse  inclusion, let $y\in \Ind\circ\Res I$.  Then $y=\sum r_{\alpha,\beta}s_\alpha s_{\beta^*}$ where each $r_{\alpha,\beta}\in\Res I$, that is, $r_{\alpha,\beta}p_v\in I$ for all $v\in\Lambda^0$.  But now  $y=\sum s_\alpha (r_{\alpha,\beta}p_{s(\alpha)}) s_{\beta^*}\in I$.  
Thus $\Ind\circ\Res I=I$, and $\Ind$ is surjective.  
Since $\Ind$ is  injective by \eqref{res-ind1}, and since $M_1\subset M_2$ if and only $\Ind M_1\subset \Ind M_2$, it follows that $\Ind$ is a lattice isomorphism. 
\end{proof}

\begin{proof}[Proof of Theorem~\ref{simplicity}]
First suppose that $\KP_R(\L)$ is simple.  Then $\KP_R(\L)$ is basically simple, and hence $\L$ is aperiodic and cofinal by Theorem~\ref{basicallysimple}. Let $M$ be a nonzero ideal of $R$.  Then $\Ind M$ is a nonzero ideal of $\KP_R(\L)$, and hence $\Ind M=\KP_R(\L)$.  By Proposition~\ref{res-ind}\eqref{res-ind1}, $M=\Res\circ\Ind M=\Res \KP_R(\Lambda)=R$.  Thus $R$ is a field.

Conversely, assume that $\Lambda$ is aperiodic and cofinal, and that $R$ is a field.  Let $I$ be a nonzero ideal of $\KP_R(\Lambda)$.  Since $\Lambda$ is aperiodic and cofinal, by Proposition~\ref{res-ind}\eqref{res-ind2} we have $I=\Ind\circ\Res I$. Thus $\Res I$ is a nonzero ideal of $R$, and hence $\Res I=R$ since $R$ is simple.  But now $I=\Ind R=\KP_R(\Lambda)$. Thus $\KP_R(\L)$ is simple.
\end{proof}

The next result is a converse for Proposition~\ref{res-ind}\eqref{res-ind2}.

\begin{proposition}\label{equivalences} 
Let $\Lambda$ be a row-finite $k$-graph without sources and let $R$ be a commutative ring with $1$. 
Then $\L$ is aperiodic and cofinal if and only if $\Ind\circ\Res=\id$.
\end{proposition}

\begin{proof}
Proposition~\ref{res-ind}\eqref{res-ind2} is the ``only if'' half.
Suppose that  $\Ind\circ\Res=\id$. It suffices by  Theorem~\ref{basicallysimple} to prove that $\KP_R(\Lambda)$ is basically simple. So let $I$ be a nonzero basic ideal of  $\KP_R(\L)$.  Then $\Ind\circ\Res I=I$ implies that $\Res I$ is a nonzero ideal. Let $0\neq r\in\Res I$. Then $rp_v\in I$ for all $v\in\L^0$, and since $I$ is basic, $p_v\in I$ for all $v\in\L^0$, and $I=\KP_R(\Lambda)$.  Thus $\KP_R(\L)$ is basically simple, as required.
\end{proof}


\goodbreak

\section{Examples and applications}\label{sectionexample}

We begin with the easiest nontrivial example.

\begin{example}\label{morealgebras}
Let $R$ be a commutative ring with $1$.  View $\L=\N^2$ as a category with a single  object $v$, and let $d:\N^2\to \N^2$ be the identity map. Then $\L$ is the unique $2$-graph whose skeleton consists of one blue and one red loop at a single vertex. For each $n\in\N^2$ there is a unique path $n$ of degree $n$, and a Kumjian-Pask family $(P,S)$ in an $R$-algebra must satisfy
\begin{gather*}
P_v^2 = P_v=S_{n^*}S_{n}=S_nS_{n^*},\\
S_mS_n = S_{m+n},\ S_{n^*}S_{m^*} = S_{(m+n)^*},\\ 
P_vS_n = S_n = S_nP_v,\ 
P_vS_{n^*} = S_{n^*} = S_{n^*}P_v.
\end{gather*}
For $q\geq m\vee n$ in $\N^2$, the sum in Lemma~\ref{pathprods} has exactly one term, and we have $S_{m^*}S_{n}= S_{q-m}S_{(q-n)^*}$;
taking  $q = m+n$ gives $S_{m^*}S_{n} = S_{n}S_{m^*}$. In particular, $\KP_R(\Lambda)$ is commutative.
We will use the graded-uniqueness theorem to show that $\KP_R(\L)$ is isomorphic to  the ring $R[x, x^{-1}, y, y^{-1}]$ of
Laurent polynomials over $R$ in two commuting indeterminates $x$ and $y$.

Set $Q_v=1$, $T_{(i,j)}=x^iy^j$ and $T_{(i,j)^*}=x^{-i}y^{-j}$.  Then $(Q,T)$ is a Kumjian-Pask $\Lambda$-family in $R[x, x^{-1}, y, y^{-1}]$, and the universal property of $\KP_R(\Lambda)$ gives a homomorphism $\phi:\KP_R(\Lambda)\to R[x, x^{-1}, y, y^{-1}]$ such that $\phi\circ p=Q$ and $\phi\circ s= T$. The groups $A_{(i,j)}:=\lsp\{x^iy^j\}$ for $(i,j)\in\Z^2$  grade $R[x,x^{-1},y,y^{-1}]$ over
$\Z^2$, and $\phi$ maps $\KP_R(\L)_{(i,j)}=\lsp\{s_ns_{m^*}:n-m=(i,j)\}$ into
$A_{(i,j)}$, so $\phi$ is graded. Finally, $\phi(rp_v)=r\phi(p_v)=r1=r\neq 0$ for all $r\in R\setminus \{0\}$, and so Theorem~\ref{gradeduniquess} implies that $\phi$ is injective.  Since the image of $\phi$ contains a generating set for $R[x,x^{-1},y,y^{-1}]$, $\phi$ is an isomorphism.
 \end{example}

\begin{remark}\label{newexamples} Let $K$ be a field.  We claim that $K[x,x^{-1},y,y^{-1}]$ cannot be realized as a Leavitt path algebra $L_K(E)$ for any directed graph $E$. Thus  Example~\ref{morealgebras} shows that the class of Kumjian-Pask algebras over $K$ is larger than  the class of Leavitt path algebras over $K$.   To see the claim, recall from  \cite[Proposition 2.7]{AC} that every commutative Leavitt path algebra has
the form $(\bigoplus_{i\in I} K) \oplus (\bigoplus_{j\in J} K[x,x^{-1}])$. Since $K[x,x^{-1},y,y^{-1}]$
has no zero divisors, if  $K[x,x^{-1},y,y^{-1}]$ had this form then it would be isomorphic to either $K$
or $K[x,x^{-1}]$ as rings.  But both $K$ and  $K[x,x^{-1}]$ are principal ideal domains, whereas $K[x,x^{-1},y,y^{-1}]$ is not. So $K[x,x^{-1},y,y^{-1}]$ is not the Leavitt path algebra of any directed graph. \end{remark}

\subsection{The Kumjian-Pask algebra and the $C^*$-algebra}\label{discussuongiut} We have said that the graded-uniqueness theorem is an analogue for Kumjian-Pask algebras of the gauge-invariant uniqueness theorem for graph $C^*$-algebras. Indeed, an original motivation for graded-uniqueness theorems was to prove that the Leavitt path algebra $L_{\CC}(E)$ embeds in the graph $C^*$-algebra $C^*(E)$, and the proof of this inevitably uses the gauge action alongside the grading of $\KP_{\CC}(\L)$. Since the existing treatments (\cite[Corollary~1.3.3]{RMalaga} and \cite[Theorem~7.3]{T}) are on the terse side, it seems worthwhile to give a careful treatment of the analogous result for Kumjian-Pask algebras. 

When the coefficient ring $R$ is the field $\CC$, the Kumjian-Pask algebra $\KP_{\CC}(\L)$ has a complex linear involution characterized in terms of the generating Kumjian-Pask family by $(cs_\lambda s_{\mu^*})^*=\bar c s_\mu s_{\lambda^*}$ for $c\in \CC$. (To see this, we \emph{define} $a\mapsto a^*$ on $\FF_{\CC}(w(X))$ by the analogous formula on infinite sums, check that this map is an involution on $\FF_{\CC}(w(X))$, and then observe that the ideal $I$ on page~\pageref{idealI} is $^*$-closed, so the involution passes to the quotient $\KP_{\CC}(\L)$.) Thus $\KP_{\CC}(\L)$ is a $*$-algebra. 

The $C^*$-algebra $C^*(\L)$ is generated by a universal Cuntz-Krieger family $(q,t)$ of the sort described at the start of \S\ref{sec:KPfam}. It is not completely obvious that such a $C^*$-algebra exists (though you'd never guess this to look at the literature!). But if we take the $*$-algebra $A$ generated by symbols $\{q_v,t_e\}$ subject to the relations, then because the elements $q_v$ and $t_e$ are all partial isometries, every generator has norm at most $1$ in every representation of $A$ as bounded operators on Hilbert space; we can then  define a semi-norm on $A$ by
\[
\|a\|=\sup\big\{\|\pi(a)\|: \pi:A\to B(H)\text{ is a $*$-representation of $A$}\big\},
\]
mod out by the ideal of elements of norm $0$ to get a normed algebra, and complete in the norm to get a $C^*$-algebra \cite[\S1]{Black}. To see that this $C^*$-algebra is nonzero, Kumjian and Pask built a Cuntz-Krieger family on $\ell^2(\L^\infty)$ in which every generator is nonzero, so in particular each $q_v$ is nonzero in $C^*(\L)$ \cite[Proposition~2.11]{KP}.

As we saw at the start of \S\ref{sec:KPfam}, the universal Cuntz-Krieger family $(q,t)$ in $C^*(\Lambda)$ is a Kumjian-Pask family with $t_{\lambda^*}:=t_\lambda^*$. Thus there is a canonical $*$-homomorphism $\pi_{q,t}:\KP_{\CC}(\L)\to C^*(\L)$ which takes $s_\lambda s_{\mu^*}$ to $t_\lambda t_\mu^*$.

\begin{prop}\label{KPembeds}
Suppose that $\Lambda$ is a row-finite $k$-graph without sources. Then $\pi_{q,t}$ is a $*$-isomorphism of $\KP_{\CC}(\L)$ onto the $*$-subalgebra
\[
A:=\lsp\big\{t_\lambda t_\mu^*:\lambda,\mu\in\L\big\}.
\]
\end{prop}

To prove this, one reaches for the graded-uniqueness theorem. However, $C^*(\L)$ is not graded in the algebraic sense: the subspaces
\begin{equation}\label{gradeKP}
C^*(\L)_n:=\clsp\big\{t_\lambda t_{\mu}^*:d(\lambda)=d(\mu)=n\big\}
\end{equation}
satisfy $C^*(\L)_mC^*(\Lambda)_n\subset C^*(\L)_{m+n}$, and are mutually linearly independent, but they do not span $C^*(\L)$ in the usual algebraic sense (see Remark~\ref{almostgrade} below). On the other hand, we have:

\begin{lemma}\label{coregraded} 
The subspaces
\[
A_n:=\lsp\big\{t_\lambda t_{\mu}^*:d(\lambda)=d(\mu)=n\big\}
\]
form a $\Z^k$-grading for the dense subalgebra $A$ of $C^*(\L)$.
\end{lemma}  

The proof of the lemma uses the \emph{gauge action}. For a directed graph $E$, the gauge action is an action of $\TT:=\{z\in \CC:|z|=1\}$ on $C^*(E)$; for a $k$-graph, it is an action $\gamma$ of the $k$-torus $\TT^k$ on $C^*(\Lambda)$. To define $\gamma_z$ for $z\in \TT^k$, invoke the universal property of $(C^*(\L),(q,t))$ to get a homomorphism $\gamma_z:C^*(\L)\to C^*(\L)$ such that $\gamma_z(q_v)=q_v$ and $\gamma_z(s_\lambda)=z^{d(\lambda)}s_\lambda$, and check that $z\mapsto \gamma_z$ is a homomorphism into $\Aut C^*(\L)$. Then it follows from an $\epsilon/3$ argument that $\gamma$ is strongly continuous in the sense that $z\mapsto \gamma_z(a)$ is continuous for each fixed $a\in C^*(\Lambda)$. (The details of the argument are in \cite[Proposition~2.1]{R} for $k=1$, and the argument carries over.)

Next we need to integrate continuous functions $f$ on $\TT^k$ with values in a $C^*$-algebra $B$. The easiest way to do this is to represent $B$ faithfully as bounded operators on a Hilbert space $H$, prove that there is a unique bounded operator $T$ on $H$ such that $(Th\,|\,k)$ is the usual Riemann integral $\int_{\TT^k} (f(z)h\,|\,k)\,dz:=\int_{[0,1]^k}\big(f(e^{2\pi i\theta})h\,|\,k\big)\,d\theta$ for $h,k\in H$, prove that $T$ belongs to $B$, and then define $\int_{\TT^k}f(z)\,dz:=T$. The construction and its properties are described in \cite[Lemma~3.1]{R} for the case $k=1$, and the general case is similar. The integral is, for example, linear and norm-decreasing for the sup-norm on $C(\TT^k,B)$.

\begin{proof}[Proof of Lemma~\ref{coregraded}]
Since each spanning element $t_\lambda t_\mu^*$ belongs to $A_{d(\lambda)-d(\mu)}$, we can by grouping terms write every $a\in A$ as a finite sum $\sum_n a_n$ with $a_n\in A_n$. To see that the $A_n$ are independent, suppose that $a_n\in A_n$ and $\sum_n a_n=0$. Elementary calculus shows that $\int_{\TT^k}z^m\,dz$ is $1$ if $m=0$ and vanishes otherwise, and hence for $m\in\Z^k$ we have
\begin{equation}\label{calchomogcoeff}
\int_{\TT^k}z^{-m}\gamma_z(t_\lambda t_\mu^*)\,dz=\Big(\int_{\TT^k}z^{-m+d(\lambda)-d(\mu)}\,dz\Big)t_\lambda t_\mu^*
=\begin{cases}
t_\lambda t_\mu^*&\text{if $m=d(\lambda)-d(\mu)$}\\
0&\text{otherwise}.
\end{cases}
\end{equation}
We deduce from linearity of the integral that if $a_n\in A_n$, then 
\[
\int_{\TT^k}z^{-m}\gamma_z(a_n)\,dz =\begin{cases}
a_m&\text{if $m=n$}\\
0&\text{otherwise}.
\end{cases}
\]
Now integrating both sides of $\sum_n a_n=0$ against $z^{-m}\gamma_z$ shows that $a_m=0$ for all $m$. An application of Lemma~\ref{pathprods} shows that if $t_\lambda t_\mu^*\in A_m$ and $t_\alpha t_\beta^*\in A_n$  then $(t_\lambda t_\mu^*)(t_\alpha t_\beta^*)\in A_{m+n}$, so $A_mA_n\subset A_{m+n}$.
\end{proof}

\begin{proof}[Proof of Proposition~\ref{KPembeds}]
The homomorphism $\pi_{q,t}$ takes $s_\lambda s_{\mu^*}$ to $t_\lambda t_\mu^*$, hence maps $\KP_{\CC}(\L)$ onto $A$ and is graded. Since we know that each $q_v$ is nonzero, and since we are working over a field, we have $\pi(rp_v)\not=0$ for every $r\not=0$ and every $v\in \L^0$. Thus the graded-uniqueness theorem implies that $\pi_{q,t}$ is injective.
\end{proof}

\begin{remark}\label{almostgrade}
The gauge action $\gamma$ was crucial in the proof of Lemma~\ref{coregraded} when we needed to recover the component $a_m$ from the expansion $\sum_na_n$, so it is certainly connected with the grading. To see why it does not give a grading of the whole $C^*$-algebra, consider an action $\beta:\TT^k\to \Aut B$ of $\TT^k$ on a $C^*$-algebra $B$, and for each $n\in \Z^k$, let
\[
B_n:=\{b\in B:\beta_z(b)=z^nb\text{ for all $z\in \TT^k$}\}.
\]
Then $B_n$ is a closed subspace of $B$, and $E_n: b\mapsto b_n:=\int_{\TT^k}z^{-n}\beta_z(b)\,dz$ is a norm-decreasing linear operator with range $B_n$ satisfying $E_n\circ E_n=E_n$. In the proof of Lemma~\ref{coregraded}, only finitely many $a_m$ are nonzero, but in general this is not the case, and we cannot expect to recover every $b\in B$ as a finite sum of elements in the $B_n$; the subspaces $B_n$ satisfy $B_mB_n\subset B_{m+n}$, but they do not grade $B$ in the algebraic sense. They are independent (because we can recover $b_m$ from a finite sum $\sum_nb_n$ by integrating), and they do determine $b$: if $b_n=0$ for all $n$, then $b=0$. 

One way to see this last point is to represent $B$ faithfully in $B(H)$, and then for each pair $h,k\in H$,
\[
(b_nh\,|\,k)=\int_{\TT^k}z^{-n}(\beta_z(b)h\,|\,k)\,dz
\]
is the $n$th Fourier coefficient of the continuous function $z\mapsto (\beta_z(b)h\,|\,k)$. Thus if $b_n=0$ for all $n$, all the Fourier coefficients of this function vanish, which implies that $(\beta_z(b)h\,|\,k)=0$ for all $z$, $h$ and $k$; taking $z=1$ shows that $(bh\,|\,k)=0$ for all $h,k$, and $b=0$. 

This last argument illustrates the difficulty. If $f$ is smooth, then the Fourier series of $f$ converges uniformly to $f$. When $f$ is just continuous, the Fourier coefficients still determine $f$, but it is not easy to recover $f$ from its Fourier series. 
\end{remark}

\begin{remark}\label{remarkgiut}
The gauge-invariant uniqueness theorem for $C^*(\Lambda)$ says that, if  $\pi:C^*(\L)\to B$ is a homomorphism (by which we mean a $*$-homomorphism) such that $\pi(q_v)\not=0$ for all $v$, and if there is a continuous action $\beta$ of $\TT^k$ on $B$ such that $\pi\circ\gamma_z=\beta_z\circ\pi$ for every $z\in \TT^k$, then $\pi$ is injective. 

For the gauge action $\gamma$ on $C^*(\L)$, we trivially have $A_n\subset C^*(\L)_n$, and since $A$ is dense in $C^*(\L)$, the norm continuity of the map $E_n:C^*(\L)\to C^*(\L)_n$ implies that $C^*(\L)_n$ is as described in \eqref{gradeKP}. One can then check that $\pi\circ\gamma_z=\beta_z\circ\pi$ for every $z\in \TT^k$ if and only if $\pi(C^*(\L)_n)\subset B_n$ for every $n\in \Z^k$. (In the ``if'' direction, the continuity of the homomorphisms $\pi\circ\gamma_z$ and $\beta_z\circ\pi$ allows us to get away with checking equality on the dense subalgebra $A$.) So we could if we wanted reformulate the gauge-invariant uniqueness theorem to look like a graded-uniqueness theorem.
\end{remark}

\subsection{Rank-2 Bratteli diagrams} Consider a $2$-graph $\Lambda$ without sources which is a rank-2 Bratteli diagram in the sense of \cite[Definition~4.1]{PRRS}. This means that the blue subgraph $B\L:=(\L^0,\L^{e_1},r,s)$ of the skeleton is a Bratteli diagram in the usual sense, so the vertex set $\L^0$ is the disjoint union $\bigsqcup_{n=0}^\infty V_n$ of finite subsets $V_n$, each blue edge goes from some $V_{n+1}$ to $V_n$, and the red subgraph $R\L:=(\L^0,\L^{e_2},r,s)$ consists of disjoint cycles whose vertices lie entirely in some $V_n$. For each blue edge $e$ there is a unique red edge $f$ with $s(f)=r(e)$, and hence by the factorization property there is a unique blue-red path $\Ff(e)h$ such that $\Ff(e)h=fe$. The map $\Ff:\L^{e_1}\to \L^{e_1}$ is a bijection, and induces a permutation of each finite set $\L^{e_1}V_n$. We write $o(e)$ for the order of $e$: the smallest $l>0$ such that $\Ff^l(e)=e$. 

\begin{prop}\label{rank2aper}
Suppose that $\L$ is a rank-2 Bratteli diagram. If $\L$ is cofinal and $\{o(e):e\in \L^{e_1}\}$ is unbounded, then $\L$ is aperiodic.
\end{prop}

Proposition~\ref{rank2aper} follows from \cite[Theorem~5.1]{PRRS}, but since  \cite{PRRS} uses a different formulation of aperiodicity, we also have to invoke the equivalence of the different notions of aperiodicity \cite[Lemma~3.2]{RS}. However, the whole point of the finite-path formulation is that it should be easier to verify. So:

\begin{proof}[Proof of Proposition~\ref{rank2aper}]
Let $v\in \L^0$, say $v\in V_{N_1}$, and take $m\not= n$ in $\N^2$. If $m_1\not =n_1$, then any path $\lambda\in \L^{m\vee n}$ has $\lambda(m)\in V_{N_1+m_1}$ and $\lambda(n)\in V_{N_1+n_1}$, and hence satisfies the aperiodicity condition \eqref{fpaperiodic}. So we suppose that $m_1=n_1$, and without loss of generality that $n_2>m_2$. As in \cite{PRRS}, we further partition each $V_N=\bigsqcup_{i=1}^{c_N}V_{N,i}$ into the sets of vertices which lie on distinct red cycles. 

As in the proof of sufficiency in \cite[Theorem~5.1]{PRRS} (see page~158 of \cite{PRRS}), cofinality implies that there exists $N$ such that, for every $M_1\geq N$,  $v\L V_{M_1,i}$ is nonempty for all $i\leq c_{M_1}$, and such that there exist $M\geq \max(N, n_1+N_1)$, $i\leq c_M$ and $g\in V_{M,i}\L^{e_1}$ such that $o(g)\geq n_2-m_2$. Now choose $\mu\in v\L V_{M,i}$, let $\alpha$ be a red path with vertices in $V_{M,i}$, $d(\alpha)\geq (0,n_2)$, $r(\alpha)=s(\mu)$ and $s(\alpha)=r(g)$, and take $\lambda:=\mu\alpha g$. Then in particular $d(\lambda)\geq (n_1,n_2)=m\vee n$, and $r(\lambda)=v$. We then have
\[
\lambda(n+d(\lambda)-(m\vee n)-e_2,\;n+d(\lambda)-(m\vee n))=\lambda(d(\lambda)-e_2,\;d(\lambda))=g,
\]
whereas 
\begin{align*}
\lambda(m+d(\lambda)-(m\vee n)-e_2&,\;m+d(\lambda)-(m\vee n))\\
&=\lambda(d(\lambda)-(n_2-m_2+1)e_2,\;d(\lambda)-(n_2-m_2))\\
&=\Ff^{n_2-m_2}(g),
\end{align*}
which is not the same as $g$ because $o(g)>n_2-m_2$. Thus the larger segments in \eqref{fpaperiodic} cannot be equal, and we have shown that $\L$ is aperiodic.
\end{proof}

\begin{cor}\label{Bdsimple}
Suppose that $\L$ is a rank-2 Bratteli diagram and $K$ is a field. If $\L$ is cofinal and $\{o(e):e\in \L^{e_1}\}$ is unbounded, then $\KP_K(\L)$ is simple.
\end{cor}
 
\begin{proof}
Since $K$ is a field, basic simplicity is the same as simplicity, so the result follows from Proposition~\ref{rank2aper} and Theorem~\ref{basicallysimple}.
\end{proof}

Notice that in the next result we have specialized to the case $K=\CC$.

\begin{prop}\label{notpi}
Suppose that $\L$ is a rank-2 Bratteli diagram. If $\L$ is cofinal and $\{o(e):e\in \L^{e_1}\}$ is unbounded, then $\KP_\CC(\L)$ is not purely infinite in the sense of \cite{AGP}.
\end{prop}

\begin{proof}
Let $P_0:=\sum_{v\in V_0}p_v$. Since $\KP_{\CC}(\L)$ is simple by Corollary~\ref{Bdsimple}, and since the property of being purely infinite simple passes to corners \cite[Proposition~10]{AA2}, it suffices for us to prove that $P_0\KP_{\CC}(\L)P_0$ is not purely infinite. We will show that $P_0\KP_{\CC}(\L)P_0$ does not contain an infinite idempotent. Suppose it does. Then there exist nonzero idempotents $p$, $p_1$, $p_2$ and elements $x$, $y$ in $P_0\KP_{\CC}(\L)P_0$ such that
\begin{equation}\label{infinrels}
p=p_1+p_2,\quad p_1p_2=p_2p_1=0,\quad xy=p\quad \text{and}\quad yx=p_1.
\end{equation}
Choose $N\in \N$ large enough to ensure that all five elements can be written as linear combinations of elements $s_\lambda s_{\mu^*}$ for which $s(\lambda)$ and $s(\mu)$ are in $\bigcup_{n=0}^N V_n$. Then the images of these elements under the isomorphism $\pi_{q,t}$ of Proposition~\ref{KPembeds} all lie in the subalgebra of $P_0C^*(\L)P_0$ spanned by the corresponding $t_\lambda t_{\mu}^*$, which by \cite[Lemma~4.8]{PRRS} is isomorphic to $P_0C^*(\Lambda_N)P_0$, where $\L_N$ is the ``rank-2 Bratteli diagram of depth $N$'' consisting of all the paths which begin and end in $\bigcup_{n=0}^N V_n$. 

Applying the Kumjian-Pask relations shows that
\[
C^*(\Lambda_N)=\clsp\{s_\lambda s_\mu^*:s(\lambda)=s(\mu)\in V_N\}.
\]
If $s(\lambda)=s(\mu)$ and $s(\alpha)=s(\beta)$ lie on different red cycles (that is, belong to different $V_{N,i}$), then $(s_\lambda s_\mu^*)(s_\alpha s_\beta^*)=0$, and hence $C^*(\L_N)$ is the $C^*$-algebraic direct sum of the subalgebras
\[
C_{N,i}=\clsp\{s_\lambda s_\mu^*:s(\lambda)=s(\mu)\in V_{N,i}\}.
\]
The blue Kumjian-Pask relation implies that the algebras $C_{N,i}$ are unital with identity $P_i:=\sum_{\alpha\in \L^{\N e_1}V_{N,i}}s_\alpha s_\alpha^*$, and indeed $C_{N,i}=P_iC^*(\L_N)P_i$. Since $P_i$ commutes with $P_0$, we then have
\[
P_0C^*(\L_N)P_0=\bigoplus_{i=1}^{c_N}P_0C_{N,i}P_0.
\]
The elements $p$, $p_1$, $p_2$, $x$ and $y$ of $P_0(C^*(\L_N)P_0$ all have direct sum decompositions, and the summands all satisfy the relations \eqref{infinrels}; in at least one summand, the component of $p_2$ is nonzero, and then the same components of all the rest must be nonzero too. So we may assume that $p$, $p_1$, $p_2$, $x$ and $y$ all belong to $P_0C_{N,i}P_0$.

Now consider the subgraph $\L_{N,i}$ of $\L_N$ with vertex set $r(s^{-1}(V_{N,i}))$. This $2$-graph has sources, but it is locally convex in the sense of \cite{RSY}, and the gauge-invariant uniqueness theorem proved there implies that the inclusion is an isomorphism of $P_0C^*(\L_{N,i})P_0$ onto $P_0C_{N,i}P_0$. The sources in $\L_{N,i}$ all lie on a single red cycle, and hence Lemma~4.5 of \cite{PRRS} implies that $P_0C_{N,i}P_0$ is isomorphic to $M_X(C(\TT))=C(\TT, M_X(\CC))$, where $X$ is the finite set $\Lambda^{Ne_1}V_N=V_0\L^{\N e_1}V_N$. Pulling the five elements through all these isomorphisms gives us nonzero idempotents $q$, $q_1$, $q_2$ and elements $f$, $g$ in $C(\TT,M_X(\CC))$ such that
\[
q=q_1+q_2,\quad q_1q_2=q_2q_1=0,\quad fg=q\quad \text{and}\quad gf=q_1.
\]
Now let $z\in\TT$. Then the equations $f(z)g(z)=q(z)$ and $g(z)f(z)=q_1(z)$ imply that $g(z)$ is an isomorphism of $q(z)\CC^X$ onto $q_1(z)\CC^X$, so the matrices $q(z)$ and $q_1(z)$ have the same rank. On the other hand, since $q_1(z)$ and $q_2(z)$ are orthogonal, $\rank(q_1(z)+q_2(z))=\rank q_1(z)+\rank q_2(z)$. Now  $q=q_1+q_2$ implies that $\rank q_2(z)=0$ for all $z$, which contradicts the assumption that $p_2$ is nonzero. Thus there is no infinite idempotent  in $P_0\KP_{\CC}(\L)P_0$, as claimed. Thus $P_0\KP_{\CC}(\L)P_0$ is not purely infinite, and neither is $\KP_{\CC}(\L)$. 
\end{proof}

Rank-2 Bratteli diagrams were invented in \cite{PRRS} to prove that the dichotomy of \cite{KPR} for simple graph $C^*$-algebras does not extend to the $C^*$-algebras of higher-rank graphs. We can now use them to see that the dichotomy of \cite[Theorem~4.4]{AA3} for simple Leavitt path algebras does not extend either.

\begin{theorem}
Suppose that $\L$ is a rank-2 Bratteli diagram, that $\L$ is cofinal, and that $\{o(e):e\in \L^{e_1}\}$ is unbounded. Then $\KP_{\CC}(\L)$ is simple but is neither purely infinite nor locally matricial.
\end{theorem}

\begin{proof}
Corollary~\ref{Bdsimple} implies that $\KP_{\CC}(\L)$ is simple, and Proposition~\ref{notpi} that it is not purely infinite. To see that it is not locally matricial, consider the element $s_\mu$ associated to a single red cycle $\mu$. Since $v:=r(\mu)=s(\mu)$ receives just one red path of length $|\mu|$, namely $\mu$, the Kumjian-Pask relation (KP4) at $v$  for $n=|\mu|e_2$ (which only involves red paths) says that $p_v=s_\mu s_\mu^*$. Thus if $E$ is the directed graph consisting of a single vertex $w$ and a single loop $e$ at $w$ and $(p,s)$ is the universal Kumjian-Pask $\Lambda$-family in $\KP_{\CC}(\L)$, then there is a homomorphism $\pi$ of the Leavitt path algebra $L_{\CC}(E)$ into $\KP_{\CC}(\L)$ which takes $w$ to $p_w$, $e$ to $s_\mu$ and ${e^*}$ to $s_{\mu^*}$. Since the image algebra $A$ is graded by $A_m:=A\cap \KP_{\CC}(\L)_{m|\mu|}=\lsp\{\mu^m\}$, and since $p_w\not=0$, the graded-uniqueness theorem for ordinary graphs implies that $\pi$ is injective. But $e$ generates the infinite-dimensional algebra $L_{\CC}(E)=\CC[x,x^{-1}]$, so $s_{\mu}$ does not lie in a finite-dimensional subalgebra.
\end{proof}

\begin{remark}
The main examples of rank-2 Bratteli diagrams are the families $\{\Lambda_\theta:\theta\in (0,1)\setminus\QQ\}$  in \cite[Example~6.5]{PRRS} and $\{\L(\mathbf{m}):\mathbf{m}\text{ is supernatural}\}$ in \cite[Example~6.7]{PRRS}. These provide models for two important families of $C^*$-algebras called the irrational rotation algebras $A_\theta$ and the Bunce-Deddens algebras $\BD(\mathbf{m})$. That their  $C^*$-algebras satisfy $C^*(\Lambda_\theta)\cong A_\theta$ and $C^*(\L(\mathbf{m}))\cong \BD(\mathbf{m})$ is proved in \cite{PRRS} by showing that the graph algebras are A$\TT$-algebras with real rank zero, hence fall into the class of $C^*$-algebras covered by a classification theorem of Elliott \cite{Ell}, computing their $K$-theory, and comparing this $K$-theory with the known $K$-theory of $A_\theta$ and $\BD(\mathbf{m})$. So the proofs will not carry over to Kumjian-Pask algebras.
\end{remark}

\end{document}